\documentclass[12pt]{article}
\usepackage{amsmath, amsthm}
\usepackage{amsfonts}
\usepackage{amssymb}
\usepackage[usenames,dvipsnames]{color}
\usepackage{graphicx}
\usepackage{hyperref}
\usepackage{bm}
\usepackage{listings}

\usepackage{color,calc}

\makeatletter
\definecolor{shade}{gray}{0.8}
        {
          \raggedright
        \setlength{\rightmargin}{\leftmargin}
        \setlength{\itemsep}{-12pt}
        \setlength{\parsep}{20pt}
        \begin{lrbox}{\@tempboxa}%
        \begin{minipage}{\linewidth-2\fboxsep}
        }
        {
        \end{minipage}%
        \end{lrbox}%
        \fcolorbox{black}{shade}{\usebox{\@tempboxa}}\newline
        }%
\makeatother

\newtheorem{proposition}{Proposition}
\newtheorem{theorem}{Theorem}
\newtheorem{lemma}{Lemma}

\renewcommand{\eqref}[1]{\hyperref[#1]{(\ref*{#1})}}

\newtheorem{remark}{Remark}
\newtheorem{definition}{Definition}

\newcommand*{\pref}[1]{\hyperref[#1]{(\ref*{#1})}}
\newcommand*{\refpref}[2]{\hyperref[#2]{\ref*{#1}(\ref*{#2})}}

\newcommand{\p}{\mathbb{P}}
\newcommand{\e}{\mathbb{E}}
\newcommand{\ud}{\mathrm{d}}

\newcommand{\R}{\mathbb{R}}
\newcommand{\N}{\mathbb{N}}
\newcommand{\Indi}[1]{\mathbf{1}_{\{#1\}}}
\newcommand{\Ind}{\mathbf{1}}
\newcommand{\mbf}[1]{\mathbf{#1}}
\newcommand{\mca}[1]{\mathcal{#1}}
\newcommand{\Exp}[1]{\exp\left\{#1\right\}}

\title{Extinction properties of multi-type continuous-state branching processes}

\author{  Andreas E. Kyprianou\thanks{Department of Mathematical Sciences, University of Bath, Claverton Down, Bath, BA2 7AY, UK. Email: \texttt{a.kyprianou@bath.ac.uk}} \ and Sandra Palau\thanks{Department of Mathematical Sciences, University of Bath, Claverton Down, Bath, BA2 7AY, UK. Email: \texttt{sp2236@bath.ac.uk}}
}  
\begin{document}

\maketitle

\begin{abstract}
\noindent Recently  in  \cite{BLP}, the notion of a multi-type continuous-state branching process (with immigration) having $d$-types was introduced as a solution to an $d$-dimensional vector-valued SDE. Preceding that, work on affine processes, originally motivated by mathematical finance, in  \cite{DFS} also showed the existence of such processes. See also more recent contributions in this direction due to  \cite{GT} and \cite{CGB}. Older work on multi-type continuous-state branching processes is more sparse but includes   \cite{W} and  \cite{M}, where  only two types are  considered.  In this paper we take a completely different approach and  consider multi-type continuous-state branching process, now allowing for up to a countable infinity of types, defined instead as a super Markov chain with both local and non-local branching mechanisms. In the spirit of \cite{EK} we explore their extinction properties and pose a number of open problems.

\medskip

\noindent {\bf Key words:} Continuous-state branching process, superprocess, multi-type process, local extinction.
\medskip

\noindent {\bf MSC 2000:}  60G99, 60J68, 60J80.
\end{abstract}

\section{Introduction}

Continuous-state branching processes (CSBP) can be seen as high density limits of Bie\-nay\-m\'e--Galton--Watson (BGW) processes. Thanks to their importance as prototypical conti\-nuum (both in space and time) asexual population models, they have been the subject of intensive study since their introduction by Ji\v{r}ina \cite{J}. For a general background on CSBPs see Chapter 12 of \cite{K} or Chapter 3 of \cite{L}, see also the review article of Caballero et al. \cite{CLU}. 

By analogy with multi-type BGW processes,  a natural extension of the class of CSBPs would be to consider a multi-type Markov population model in continuous time which exhibits a branching property. Indeed, in whatever sense they can be defined, multi-type CSBPs (MCSBP) should have the property that the continuum mass of each type reproduces within its own population type in a way that is familiar to a CSBP, but also allows for the migration and/or seeding of mass  into other population types.

Recently  in  \cite{BLP}, the notion of a multi-type continuous-state branching process (with immigration) having $d$-types was introduced as a solution to an $d$-dimensional vector-valued stochastic differential equation (SDE) with both Gaussian and Poisson driving noise. Preceding that, work on affine processes, originally motivated by mathematical finance, in \cite{DFS} also showed the existence of such processes. See also more recent contributions in this direction due to  \cite{GT} and \cite{CGB}. Older work on multi-type continuous-state branching processes is more sparse but includes  \cite{W} and  \cite{M}, where  only two types are  considered.

In this article, we introduce MCSBPs through the medium of  super Markov chains. That is to say we defined MCSBPs as superprocesses whose associated underlying Markov movement generator is that of a Markov chain.
This allows us the possibility of working with a countably infinite number of types. We are interested in particular in the event of extinction and growth rates. Lessons learnt from the setting of super diffusions tells us that, in the case that the number of types is infinite, we should expect to see the possibility that the  total mass may grow arbitrarily large whilst the population of each type dies out; see for example the summary in  Chapter 2 of  \cite{E}. This type of behaviour can be attributed to the notion of transient `mass transfer' through the different types and is only possible with an infinite number of types. In the case that the number of types is finite, we know from the setting of multi-type Bienaym\'e--Galton--Watson  processes (MBGW) that all types grow at the same rate and we expect the same to be true of MCSBPs.

\section{Main results}

Our first main result is to identify the existence of MCSBPs, allowing for   up to a countable infinity of types. Denote by $\N=\{1,2,\cdots\}$ the natural numbers. Let
$\mca{B}(\N)$ be the space of bounded measurable functions on $\N$. Thinking of a member of $\mca{B}(\N)$, say ${f}$, as a vector we will write its entries by $f(i)$, $i\in\mathbb{N}$.   Write $\mathcal{M}(\N)$ the space of finite Borel measures on $\N$, let  $\mca{B}^+(\N)$ the subset of bounded positive functions.

\begin{theorem} \label{MCSBP}
Suppose that 

\begin{align}\label{defpsi}
\psi(i,z)=b(i)z+c(i)z^2+\int_0^{\infty}({\rm e}^{-zu}-1+zu)\ell(i,\ud u), \qquad i\in\N,\quad z\geq0,
\end{align}
where ${b}\in \mca{B}(\N)$, ${c}\in \mca{B}^+(\N)$ and, for each $i\in\N$, $(u\wedge u^2)\ell(i,\ud u)$ is a bounded kernel from $\N$ to $(0,\infty)$. 
Suppose further that 
\begin{align}\label{defzeta}
\phi(i,{f})=-\beta(i)\left[d(i)\langle {f},\, {\pi}_i\rangle+\int_0^{\infty} (1-{\rm e}^{-u\langle {f},\, {\pi}_i\rangle}){\rm n}(i,\ud u)\right], \qquad i\in\N, f\in \mathcal{B}^+(\N),
\end{align}
where $d, \beta\in \mca{B}^+(\N)$, ${\pi}_i$ is a probability distribution on $\N\backslash\{i\}$ (specifically $\pi_i(i) = 0$, $i\in\N$) and, for $i\in\mathbb{N}$, $u {\rm n}(i,\ud u)$ is a bounded kernel from $\N$ to $(0,\infty)$ with
$$d(i) +\int_0^{\infty} u{\rm n}(i,\ud u)\leq 1.$$
Then there exists an $[0,\infty)^\N$-valued strong Markov process ${X}: = ({X}_t,t\geq 0)$, where ${X}_t = (X_t{(1)}, X_t{(2)}, \cdots)$, $t\geq 0$, with probabilities $\{\mbf{P}_{\mu}, {\mu}\in  \mathcal{M}(\N)\}$ such that 
\begin{align}\label{semigroup v}
\mbf{E}_{\mu}[{\rm e}^{-\langle {f},{X}_t\rangle}]=\Exp{-\langle V_t{f},{\mu}\rangle},\quad {\mu}\in  \mathcal{M}(\N), \ {f}\in \mca{B}^+(\N),
\end{align}
where, for $i\in\N$,
\begin{align}\label{ecv}
V_tf(i)=f(i)-\int_0^t\Big[\psi(i,V_{s}{f}(i))+\phi(i,V_{s}{f})\Big] \ud s, \qquad t\geq 0.
\end{align}
\end{theorem}
In the above theorem, for ${f}\in\mathcal{B}^+(\N)$ and ${\mu}\in\mathcal{M}(\N)$, we use the notation 
\[
\langle {f},{\mu}\rangle: = \sum_{i\geq 1} f(i)\mu(i).
\]
Equation (\ref{semigroup v}) tells us that $X$ satisfies the branching property: for $\mu_1, \mu_2\in  \mathcal{M}(\N)$, 
\[
\mbf{E}_{\mu_1+\mu_2}[{\rm e}^{-\langle {f},{X}_t\rangle}]  = \mbf{E}_{\mu_1}[{\rm e}^{-\langle {f},{X}_t\rangle}] \mbf{E}_{\mu_2}[{\rm e}^{-\langle {f},{X}_t\rangle}], \qquad t\geq0.
\]
That is to say, $(X, \mbf{P}_{\mu_1+\mu_2})$ is equal in law to the sum of independent copies of $(X, \mbf{P}_{\mu_1})$ and $(X, \mbf{P}_{\mu_2})$. We can also understand the process $X$ to be the natural multi-type generalisation of a CSBP as, for each type $i\in\N$, $X{(i)}$ evolves, in part from a local contribution which is that of a CSBP with mechanism $\psi(i,z)$, but also from a non-local contribution from other types. The mechanism $\phi(i,\cdot)$ dictates how this occurs. Roughly speaking, each type $i\in\N$ seeds an infinitesimally small mass continuously at rate $\beta(i)d(i)\pi_i(j)$ on to sites $j\neq i$ (recall $\pi_i(i) = 0$, $i\in\N$). Moreover, it seeds an amount of mass $u>0$ at rate $\beta(i){\rm n}(i, \ud u)$ to sites $j\neq i$ in proportion given by $\pi_i(j)$. We refer to the processes described in the above theorem as $(\psi,\phi)$ multi-type continuous-state branching processes, or $(\psi,\phi)$-MCSBPs
 for short.
\medskip

Our main results concern how the different types of extinction occur for a MCSBP $X$ as defined above. As alluded to in the introduction, we must distinguish {\it local extinction} at a finite number of sites $A\subset \mathbb{N}$, that is,
\[
\mathcal{L}_A: =\{\lim_{t\to\infty} \langle\mathbf{1}_A, X_t\rangle =0\},
\]
 from global extinction of the process $X$, i.e. the event 
\[
\mathcal{E}: = \{\lim_{t\to\infty}\langle1, {X}_t\rangle =0\}.
\]
The distinction between these two has been dealt with in the setting of super diffusions by \cite{EK}. In this article, we use techniques adapted from that paper to understand local extinction in the setting here. 
The case of global extinction can be dealt with in a familiar way. To this end, denote by $\delta_i$ the atomic measure consisting of a unit mass concentrated at point $i\in \N$.

\begin{lemma}\label{global}For each $i\in\N$, let ${w}$ be the vector with entries
$w(i):=-\log \mbf{P}_{\delta_i}(\mca{E})$, $i\in \mathbb{N}$.
Then  ${w}$ is a non-negative solution to
\begin{equation}
\psi(i,w(i))+\phi(i,{w})=0, \qquad i\in\N.
\label{fixedpoint}
\end{equation} 
\end{lemma}

For the case of local extinction, a more sophisticated notation is needed. First we must introduce the notion of the linear semigroup. For each ${f}\in\mathcal{B}(f) $, define the linear semigroup $({\mathcal M}_t,t\geq 0)$ by
$${\mathcal M}_tf(i):=\mbf{E}_{\delta_i}[\langle {f},{X}_t\rangle ], \qquad t\geq 0.$$
Define the matrix ${M}(t)$ by
$${M}(t)_{ij}:=\mbf{E}_{\delta_i}[X_t(j)], \qquad t\geq 0,$$
and observe that
${\mathcal M}_t[f](i)=[{M}(t)f](i),$ for $t\geq 0$ and ${f}\in\mathcal{B}(\N)$.
The linear semigroup and its spectral properties play a crucial role in determining the limit behavior of the MCSBP.
In what follows, we need to assume that ${M}(t)$ is \textit{irreducible} in the sense that, for any $i,j\in\N$, there exists $t>0$ such that $M(t)_{ij}>0$. To this end, we make the following global assumption throughout the paper, which ensures irreducibility of ${M}(t)$, $t\geq 0$.

\bigskip

\noindent {\bf (A):} {\it The matrix $\pi_i(j)$, $i,j\in\N$, is the transition matrix of an irreducible Markov chain.}

\bigskip

For each $i,j\in \N$, and $\lambda\in \R$ we define the matrix $H(\lambda)$ by 
$${H}_{ij}(\lambda):=\int_0^{\infty} {\rm e}^{\lambda t} {M}(t)_{ij} \ud t.$$ 
The following result is the analogue of a result proved for linear semigroups of MBGW processes; see e.g. Niemi and Nummelin (\cite{ninu}, Proposition 2.1) or Lemma 1 of \cite{Moy}. In light of this, its proof is straightforward omitted for the sake of brevity.
\begin{lemma}\label{spectralradius} If, for some $
\lambda$,  ${H}_{ij}(\lambda)<\infty$ for a pair $i,j$, then ${H}_{ij}(\lambda)<\infty$ for all $i,j\in\N$.
In particular, the parameter 
$$\Lambda_{ij}=\sup\{\lambda\geq -\infty: {H}_{ij}(\lambda)<\infty\},$$
 does not depend on $i$ and $j$. The common value, $\Lambda=\Lambda_{ij}$, is called the spectral radius of $M$.  
\end{lemma} 

In contrast to Lemma \ref{global}, which shows that global extinction depends on the initial configuration of the MCSBP through the non-linear functional fixed point equation (\ref{global}), case of local extinction on any finite number of states depends only on the spectral radius $\Lambda$. In particular local extinction for finite sets is not a phenomenon that is set-dependent.
 
\begin{theorem}[Local extinction dichotomy]\label{localextintion} Fix ${\mu}\in \mathcal{M}(\mathbb{N})$ such that $\sup\{n: \mu(n)>0\}<\infty$. Moreover suppose that 
\begin{equation}
\int_1^\infty (x\log x) \ell(i,{\rm d}x)  + \int_1^\infty (x\log x) {\rm n}(i,{\rm d}x) <\infty, \qquad \mbox{ for all } i\in \N
\label{globalxlogx}
\end{equation}
holds.
\begin{description}
\item[(i)] For any finite number of  states $A\subseteq \mathbb{N}$,  $\mbf{P}_{\mu}(\mathcal{L}_A) =1 $ if and only if $\Lambda\geq 0$. 
\item[(ii)] For any finite number of states $A\subseteq \mathbb{N}$,  let ${v}_A$ be the vector with entries $v_A(i) = -\log \mathbf{P}_{\delta_i}(\mathcal{L}_A)$, $i\in\mathbb{N}$, Then $v_A$ is a solution to \eqref{fixedpoint}, and $v_A(i)\leq w(i)$ for all $i\in\N$. 
\end{description}
\end{theorem}

\begin{remark}
 As we will see in the proof, if $\Lambda\geq 0$, then the process has local extinction a.s. even if \eqref{globalxlogx} is not satisfied.
\end{remark}

The results in this paper open up a number of questions for the MCSBP which are motivated by similar issues that emerge in the setting of CSBPs and  super diffusions.  For example, by analogy with the setting for super diffusions, under the assumption (\ref{globalxlogx}), we would expect that when $\Lambda<0$, the quantity $-\Lambda$ characterises the growth rate of individual types. Specifically we conjecture that, when local extinction fails,  $\exp\{\Lambda t\}X_t(i)$ converges almost surely to a non-trivial limit as $t\to\infty$, for each $i\in\N$. Moreover,  if the number of types is finite, then $-\Lambda$ is also the growth rate of the total mass. That is to say  $\exp\{\Lambda t\}\langle 1, X_t\rangle$ converges almost surely to a non-trivial limit as $t\to\infty$. If the total number of types is infinite then one may look for a  discrepancy between the global growth rate and local growth rate. In the setting of super diffusions, \cite{ERS} have made some progress in this direction. Referring back to classical theory for CSBPs, it is unclear how the event of extinction occurs, both locally and globally. Does extinction occur as a result of mass limiting to zero but remaining positive for all time, or does mass finally disappear after an almost surely finite amount of time? Moreover, how does the way that extinction occur for one type relate to that of another type? An irreducibility property of the type space, e.g. assumption (A), is likely to ensure that mass in all states will experience extinction in a similar way with regard to the two types of extinction described before, but this will not necessarily guarantee that global extinction behaves in the same way as local extinction. We hope to address some of these questions in future work.

\bigskip
 
We complete this section by giving an overview of the remainder of the paper. In the next section we give the construction of MCSBPs as a scaling limit of MBGW processes; that is to say, in terms of  branching Markov chains. We define the linear semigroup associated to the MCSBP. The so-called spectral radius of this linear semigroup will have an important role in the asymptotic behaviour of our process, in particular, it will determine the phenomenon of local extinction. The properties of the linear semigroup are studied in Section \ref{msemigroup1}. In Sections \ref{spine1} and \ref{spine2} we develop some standard tools based around a spine decomposition. In this setting, the spine is a Markov chain and we note in particular that the non-local nature of the branching mechanism induces a new additional phenomenon in which a positive, random amount of mass immigrates off the spine  each time it jumps from one state to another. Moreover, the distribution of the immigrating mass depends on where the spine jumped from and where it jumped to. Concurrently to our work we learnt that this phenomenon was also observed recently by Chen, Ren and Song \cite{CRS}.
In Section \ref{proofs}, we give the proof of the main results. We note that the main agenda for the proof was heavily influenced by the proof of local extinction in \cite{EK} for super diffusions. Finally in Section \ref{sect:examples}, we provide examples to illustrate the local phenomenon property.

\section{MCSBPs as a superprocess}

Our objective in this section  is to prove Theorem \ref{MCSBP}.  The proof is not novel as we do this by showing that MCSBPs can be seen in, in the spirit of the theory of superprocesses, as the scaling limits of MBGW processes  with type space $\N$ (or just $\{1,\cdots, n\}$ for some $n\in\N$ in the case of finite types). 

To this end, let $\gamma\in \mca{B}^+(\N)$  and let $F(i,\ud\nu)$ be a Markov kernel from $\N$ to  $\mathcal{I}(\N)$, the space of finite integer-valued  measures, such that
\begin{align*}
\underset{i\in\N}{\sup}\int_{\mathcal{I}(\N)} \nu(1) F(i,\ud\nu)<\infty.
\end{align*}
A branching particle system is described by the following properties:
\begin{enumerate}
	\item For a particle of type  $i\in\N$, which is alive at time $r\geq 0$, the conditional probability of survival during the time interval $[r,t)$ is $\rho_i(r,t):=\exp\{-\gamma(i)(t-r)\}$, $t\geq r$.
	\item When a particle of type $i$ dies, it gives birth to a random number of offspring in $\N$ according to the probability kernel $F(i,\ud\nu)$. 
 
\end{enumerate}

We also assume that the lifetime and the branching of different particles are independent. Let $X_t(B)$
denote the number of particles in $B\in\mathcal{B}(\N)$ that are alive at time $t\geq 0$ and assume $X_0(\N)<\infty$. 
With a slight abuse of notation, we  take $X_0 : = \mu$, where ${\mu}\in\mathcal{I}(\N)$.
Then $\{X_t:t\geq 0\}$ is a Markov process with state space $\mathcal{I}(\N)$, which will be referred as a \textit{branching Markov chain} or \textit{multi-type BGW} with parameters $(\gamma,F)$.
For ${\mu}\in \mathcal{I}(\N)$, let $\mbf{P}_{\mu}$ denote the  law of $\{X_t: t\geq 0\}$ given $X_0=\mu$. In the special case that $X$ is issued with a single particle of type $i$, we write its law by $\mathbf{P}_{\delta_i}$. For ${f}\in B^+(\N)$, $t\geq 0$, $i\in\N$, put
$$u_t(i):=u_t(i,f)=-\log \mbf{E}_{\delta_i}[\Exp{-\langle {f},{X}_t\rangle}].$$
The independence hypothesis implies that
\begin{align}\label{semigrupo u}
\mbf{E}_{\mu}[\Exp{-\langle {f},{X}_t\rangle}]=\Exp{-\langle u_t,{\mu}\rangle}, \qquad {\mu}\in\mathcal{I}(\N),\ {f}\in\mathcal{B}^+(\N),\ t\geq 0.
\end{align}
Moreover, by conditioning on the first branching event, $u_t$ is determined by the renewal equation
$${\rm e}^{-u_t(i)}=\rho_i(0,t){\rm e}^{-f(i)}+\int_0^t \rho_i(0,s)\gamma(i)\int_{\mathcal{I}(\N)}{\rm e}^{-\langle u_{t-s},\ \nu\rangle }F(i,\ud\nu) \ud s.$$
By a standard argument (see for example Lemma 1.2 in Chapter 4 of  in \cite{dyn02}) one sees that the last equation is equivalent to
\begin{align}\label{ut}
{\rm e}^{-u_t(i)}=&{\rm e}^{-f(i)}-\int_0^t \gamma(i){\rm e}^{-u_{t-s}(i)} \ud s+\int_0^t \gamma(i)\int_{\mathcal{I}(\N)}{\rm e}^{-\langle u_{t-s},\ \nu\rangle}F(i, \ud\nu)\ud s.
\end{align}
 See, for example, Asmussen and Hering \cite{ashe} or Ikeda et al. \cite{iknawa1,iknawa2,iknawa3} for similar constructions.

 In preparation for our scaling limit, it is convenient to treat the offspring that start their motion from the death sites of their parents separately from others. To this end, we introduce some additional parameters. Let $\alpha$ and $\beta\in \mca{B}^+(\N)$ such that $\gamma=\alpha+\beta$. For each $i\in \N$, let ${\pi}_i$ be a probability distribution in $\N\setminus\{i\}$ and let
$g, h$ be two positive measurable functions from $\N\times [-1,1]$ to $\mathbb{R}$ such that, for each $i\in\N$,
$$g(i,z)=\underset{n=0}{\overset{\infty}{\sum}}p_n(i)z^n,\qquad h(i,z)=\underset{n=0}{\overset{\infty}{\sum}}q_n(i)z^n  \qquad |z|\leq 1,$$
are probability generating functions with $\sup_{i}g_z'(i,1-)<\infty$ and 
$\sup_{i}h_z'(i,1-)<\infty$.
Next, define the probability kernels $F_0(i,d\nu)$ and $F_1(i,d\nu)$ from $\N$ to $\mathcal{I}(\N)$ by
$$\int_{\mathcal{I}(\N)}{\rm e}^{-\langle {f},\ \nu\rangle }F_0(i,\ud\nu)=g(i, {\rm e}^{-f(i)})$$
and
$$\int_{\mathcal{I}(\N)}{\rm e}^{-\langle {f},\ \nu\rangle }F_1(i,\ud\nu)=h(i,\langle {\rm e}^{-f}, {\pi}_i\rangle).$$
 We replace the role of  $F(i,d\nu)$ by
$$\gamma^{-1}(i)\left[\alpha(i)F_0(i,\ud\nu)+\beta(i)F_1(i,\ud\nu)\right], \qquad i\in\mathbb{N}, \nu\in\mathcal{I}(\N).$$

Intuitively, when a particle of type   $i\in\N$ splits, the branching is of local type with probability $\alpha(i)/\gamma(i)$ and is of non-local type with probability $\beta(i)/\gamma(i)$. If branching is of  a local type, the distribution of the offspring number is $\{p_n(i)\}$. If branching is of a non-local type, the particle gives birth to a random number of offspring according to the distribution $\{q_n(i)\}$, and those offspring choose their locations in $\N\setminus\{i\}$ independently of each other according to the distribution $\pi_i(\cdot)$. Therefore, $u_t$ is determined by the renewal equation
\begin{align}\label{ut2}
{\rm e}^{-u_t(i)}={\rm e}^{-f(i)}&+\int_0^t \alpha(i)\Big[g(i,{\rm e}^{-u_{t-s}(i)})-{\rm e}^{-u_{t-s}(i)} \Big]
 \ud s\nonumber\\ 
 &+\int_0^t \beta(i)\Big[h(i,\langle {\rm e}^{-u_{t-s}}, \pi_{i}\rangle)-{\rm e}^{-u_{t-s}(i)} \Big] \ud s.
\end{align}
For the forthcoming analysis, it is more convenient to work with 
$$v_t(i):=v_t(i,f)=1-\exp\{-u_t(i,f)\}, \qquad t\geq0, i\in\mathbb{N}.$$
In that case, 
\begin{align*}
v_t(i)=\e_i\left[1-{\rm e}^{f(i)}\right]-\int_0^t\left[ \psi(i,v_{t-s}(i))+\phi(i,v_{t-s})\right]\ud s, 
\end{align*}
where 
$$\psi(i,z)=\alpha(i)[g(i,1-z)-(1-z)]+\beta(i)z$$
and
$$\phi(i,f)=\beta(i)\left[h(i,1-\langle {f},{\pi}_i\rangle)-1\right].$$


Next, we take a scaling limit of the MBGW process. We treat the limit as a superprocess with local and non-local branching mechanism. For each $k\in\N$, let $\{Y^{(k)}(t),t\geq 0\}$ be a sequence of branching particle system determined by $(\alpha_k(\cdot),\beta_k(\cdot),g_k(\cdot),h_k(\cdot),\pi_\cdot)$. Then, for each $k$,
$$\{X^{(k)}(t)=k^{-1}Y^{(k)}(t),\quad t\geq 0\}$$
defines a Markov process in $N_k(\N):=\{k^{-1}\sigma, \sigma\in \mathcal{I}(\N)\}$. For $0\leq z\leq k$ and ${f}\in \mca{B}(\N)$, let
$$\psi_k(i,z)=k\alpha_k(i)[g_k(i,1-z/k)-(1-z/k)]+\beta_k(i)z$$
and
$$\phi_k(i,f)=\beta_k(i)k[h_k(i,1-k^{-1}\langle {f},{\pi}_i\rangle)-1].$$
Let denote by $u_t^k(i,f)=-\log \mbf{E}_{\delta_i}[\Exp{-\langle {f},{X}_t^k\rangle}]$ and $v_t^k(i,f)=1-\exp\{-u_t^k(i,f)\}$.

Under certain conditions, Dawson et. al \cite{dagoli} obtained the convergence of $\{X^{(k)}(t), t\geq 0\}$ to some process $\{X(t), t\geq 0\}$.   Let  $\overline{\mca{B}}(\N)$ be the subset ${\mca{B}}(\N)$ with entries uniformly bounded from above and below. We re-word their result for our particular setting here.

\begin{theorem}\label{Dawson}
Suppose that
$$\underset{n=0}{\overset{\infty}{\sum}}nq_n^{k}(i)\leq 1,$$
that $\beta_k \rightarrow \beta\in \mca{B}^+(\N)$ uniformly, $\phi_k(i,f)\rightarrow \phi(i,f)$ uniformly on $\N\times\overline{\mca{B}}(\N)$, and $\psi(i,z)\rightarrow\psi(i,z)$ locally uniformly. Then
\begin{enumerate}
\item[i)] The function $\psi(i,z)$ has representation
\begin{align}\label{Ddefphi}
\psi(i,z)=b(i)z+c(i)z^2+\int_0^{\infty}({\rm e}^{-zu}-1+zu)\ell(i,\ud u), \qquad i\in\N,\quad z\geq0,
\end{align}
where $b\in \mca{B}(\N)$, $c\in \mca{B}^+(\N)$ and $(u\wedge u^2)\ell(i,\ud u)$ is a bounded kernel from $\N$ to $(0,\infty)$.

\item[ii)] The function $\phi(i,f)$ can be represented as
\begin{align}\label{Ddefpsi}
\phi(i,f)=-\beta(i)\left[d(i)\langle {f},\, {\pi}_i\rangle+\int_0^{\infty} (1-{\rm e}^{-u\langle {f},\, {\pi}_i\rangle}){\rm n}(i,\ud u)\right], 
\end{align}
where $d\in \mca{B}^+(\N)$, and $u {\rm n}(i,\ud u)$ is a bounded kernel from $\N$ to $(0,\infty)$ with
$$d(i) +\int_0^{\infty} u{\rm n}(i,\ud u)\leq 1.$$
\item[iii)] To each function $\psi$ and $\phi$ satisfying (\ref{Ddefphi}) and (\ref{Ddefpsi}) there correspond a sequence of $\beta_k$, $\psi_k$ and $\phi_k$.
\item[v)] For each $a\geq 0$, the functions $v_t^{k}(i,f)$ and $ku_t^{(k)}(i,f)$ converge boundedly and uniformly on $[0,a]\times\N\times \overline{\mca{B}}(\N)$, to the unique bounded positive solution $V_tf(i)$ to the evolution equation
\begin{align}\label{boundedevolution}
V_tf(i)=f(i)- \int_0^t\Big[\psi(i,V_{t-s}f(i))+\phi(i,V_{t-s}f)\Big]{\rm d}s, \qquad t\geq 0 .
\end{align}
\end{enumerate}

Moreover, there exists a Markov process $\{X_t:t\geq 0\}$ with probabilities $\{\mbf{P}_{\mu}, {\mu}\in  \mathcal{M}(\N)\}$ such that 
\begin{align*}
\mbf{E}_{\mu}[{\rm e}^{-\langle {f},{X}_t\rangle}]=\Exp{-\langle V_t{f},{\mu}\rangle},\quad {\mu}\in  \mathcal{M}(\N), \ {f}\in \mca{B}^+(\N),
\end{align*}
and the cumulant semigroup  $V_tf$ is given by \eqref{boundedevolution}.
\end{theorem}
Theorem \ref{MCSBP} now follows directly as a corollary of the above result.
Intuitively, $\psi(i,\cdot)$ describes the rate at which a branching event amongst current mass of type  $i\in\N$, produces further mass of type $i$. Moreover,  $\phi(i,\cdot)$ describes the rate at which a branching event amongst current mass of type  $i\in\N$, produces further mass of other types $\N\backslash\{i\}$.

\begin{remark}\rm
The non-local branching mechanism is not the most general form that can be assumed in the limit. Indeed, taking account of the class of non-local branching mechanisms that can be developed in  \cite{dagoli}, \cite{dyn04} and \cite{L}, we may do the same here. Nonetheless, we keep to this less-general class for the sake of mathematical convenience.
\end{remark}

\section{Spectral properties of the moment semigroup}\label{msemigroup1}
Let $(X_t,\mbf{P}_{\mu})$ be a MCSBP  and define its linear semigroup $({\mathcal M}_t,t\geq 0)$ by
\begin{equation}{\mathcal M}_t[f](i):=\mbf{E}_{\delta_i}[\langle {f},{X}_t\rangle ], \qquad i\in\N, {f}\in\mathcal{B}^+(\N), t\geq 0.
\label{msg}
\end{equation}
By replacing $f$ in (\ref{semigroup v}) and (\ref{ecv}) by $\lambda f$ and differentiating with respect to $\lambda$ and then setting $\lambda=0,$  we can verified that 
 $${\mathcal M}_t[f](i)=f(i)+\int_0^t  {\mathcal{K}}[\mathcal{M}_{s}[f]](i)\ud s-\int_0^tb(i){\mathcal M}_{s}[f](i)\ud s,\qquad i\in\N, {f}\in\mathcal{B}^+(\N), t\geq 0,$$
where
$${\mathcal K}[g](i)=\beta(i)\left(d(i)+\int_0^{\infty} u{\rm n}(i,\ud u)\right)\langle g,\, {\pi}_i\rangle.$$
(For similar computations see Propositions 2.24 and 2.29 in \cite{L}). Denote by $L$ the infini\-tesimal generator of ${\mathcal M}_t$,
$$L[f](i)=\underset{t\rightarrow 0}{\lim} \frac{{\mathcal M}_t[f](i)-f(i)}{t},\qquad \qquad i\in\N, {f}\in\mathcal{B}(\N).$$
Then, the operator $L[f](i)$ is the matrix product with $L$ given by
\begin{eqnarray}\label{generadorsuperchain}
{L}={\Delta}_{-b}+{K},
\end{eqnarray}
where the matrices ${\Delta}_{-b}$ and ${K}$ are given by
$$({\Delta}_{-b})_{ij}=-b(i)\Ind_{i=j}, \qquad \mbox{and}\qquad  {K}_{ij}=\beta(i)\left(d(i)+\int_0^{\infty} u{\rm n}(i,\ud u)\right)\pi_i(j).$$


Define the matrix ${M}(t)$ by
$$M(t)_{ij}:=\mbf{E}_{\delta_i}[X_t(j)], $$
and observe that
\begin{equation}{\mathcal M}_t[f](i)=[{M}(t)f](i).\label{mxsg}\end{equation}

The linear semigroup will play an important role in the proof of Theorem \ref{localextintion}, in particular, its spectral properties are of concern to us. 
Thanks to (\ref{mxsg}), it suffices to study the spectral properties of the matrix  ${M}(t)$. 
 In the forthcoming theory, we will need to assume that ${M}: = \{{M}(t) : t\geq 0\}$, is \textit{irreducible} in the sense that for any $i,j\in\N$ there exists $t>0$ such that $M(t)_{ij}>0$. The following lemma ensures this is the case.
 
 \begin{lemma}\label{irreducible}
 Suppose that $\pi_i(j)$, $i,j\in\N$ is the transition matrix of an irreducible Markov chain, then  ${M}$ is irreducible.
 \end{lemma}
 \begin{proof}
 Let $a(i)=\beta(i)\left(d(i)+\int_0^{\infty} u{\rm n}(i,\ud u)\right)$, for i$\in\N$. Define the matrices ${Q}$ and ${\Delta}_{a-b}$
 $${Q}_{ij}=a(i)(\pi_i(j)-\Indi{i=j})\qquad \mbox{and}\qquad  ({\Delta}_{a-b})_{ij}=(a(i)-b(i))\Indi{i=j}.$$
 By hypothesis, ${Q}$ is the $Q$-matrix of an irreducible Markov chain $(\xi_t, \p_i)$. In particular, for each $i,j\in \N$ and $t>0$, $\p_i(\xi_t=j)>0$.
 Observe in (\ref{generadorsuperchain}) that ${L}={Q}+{\Delta}_{a-b}$ which is the formal generator of the semigroup given by 
 \begin{align}\label{f-k}
 {\mathcal T}_t[f](i)=\e_i\left[f(\xi_t)\exp\left\{\int_0^t (a-b)(\xi_s){\rm d}s\right\}\right]\qquad i\in\N,\ {f}\in\mathcal{B}^+(\N),\ t\geq 0.
 \end{align}
  By uniqueness of semigroups, ${\mathcal M}_tf(i)={\mathcal T}_t[f](i)$, $t\geq0$, $i\in \N$. In particular $M(t)_{ij}=\mathcal{T}_t[\delta_j](i)>0$, where $\delta$ is the Dirac function, and therefore ${M}$ is irreducible.
 \end{proof}

Recall that, for each $i,j\in \N$ and $\lambda\in \R$,  we defined the matrix $H(\lambda)$ by 
$$H_{ij}(\lambda):=\int_0^{\infty} {\rm e}^{\lambda t} M(t)_{ij} {\rm d}t.$$
and that the spectral radius
$$\Lambda:=\sup\{\lambda\geq -\infty: H_{ij}(\lambda)<\infty\},$$
 does not depend on $i$ and $j$.

\begin{definition}
A non-negative vector ${x}$ with entries $x(i)$, $i\in\N$, is called \textit{right}  (resp. \textit{left}) \textit{subinvariant} $\lambda$-\textit{vector}, if for all $t\geq 0$,
$${M}(t){x}\leq {\rm e}^{-\lambda t}{x}, \qquad \qquad (\text{resp. } {x}^{T}{M}(t)\leq {\rm e}^{-\lambda t}{x}).$$
If the equality holds, the vector is call a right (resp. left) invariant $\lambda$-vector.
\end{definition}

\noindent In the next proposition, we appeal to standard techniques (cf. \cite{Moy} or \cite{se}) and provide  sufficient conditions for the existence of subinvariant $\lambda$-vectors. 

\begin{proposition}\label{subinvariant}
If ${H}(\lambda)<\infty$,  then there exists a positive\footnote{Recall that a vector ${x}$ is positive if its entries, $x(i)$, are strictly positive  for all $i$.}  right subinvariant $\lambda$-vector, ${x}$, and a positive left subinvariant $\lambda$-vector, ${y}$.
 There exists no left  or right subinvariant $\beta $-vector for $\beta>\Lambda$.
\end{proposition}

\begin{proof}Fix $j\in \N$ and define ${x}$ and ${y}$ as follows
$$x(i)=H_{ij}(\lambda) \qquad \mbox{ and }  \qquad y(i)=H_{ji}(\lambda).$$
Since the function $t\mapsto {M}(t)$ is continuous and ${M}$ is irreducible,  $x(i)y(k)>0$ for all $i,k\in \N$. Let $s\geq 0$, by Fubini's Theorem,
$$[{y}^{T}{M}(s)](i)=\underset{k\in\N}{\sum}\int_0^{\infty} {\rm e}^{\lambda t}M(t)_{jk}\ud t\ M(s)_{ki}=\int_0^{\infty} {\rm e}^{\lambda t}\underset{k\in\N}{\sum}M(t)_{jk} M(s)_{ki}\ud t.$$
The semigroup property implies that
$$[{y}^{T}{M}(s)](i) = \int_0^{\infty} {\rm e}^{\lambda t} M(s+t)_{ji}\ud t= {\rm e}^{-\lambda s}\int_s^{\infty} {\rm e}^{\lambda t} M(t)_{ji}\ud t\leq {\rm e}^{-\lambda s}y(i).$$
Therefore, ${y}$ is a left  subinvariant $\lambda$-vector. A similar computation shows that ${x}$ is a right subinvariant $\lambda$-vector.

Suppose ${x}$ is a right subinvariant $\beta$-vector. Let $\alpha\in(\Lambda, \beta)$, then, for each $i\in\N$,
$$\int_0^{\infty} {\rm e}^{\alpha t}[{M}(t){x}](i)\ud t \leq \int_0^{\infty} {\rm e}^{\alpha t}{\rm e}^{-\beta t}x(i)\ud t=x(i)(\beta-\alpha)^{-1}.$$
Let $j\in\N$ such that $x(j)>0$, then
$$\int_0^{\infty} {\rm e}^{\alpha t}M(t)_{ij}\ud t\leq \frac{x(i)}{x(j)}(\beta-\alpha)^{-1}<\infty,$$
which is a contradiction with the definition of $\Lambda$. In an analogous way, there is no left subinvariant $\beta$-vector.

\end{proof}

When $H_{ij}(\Lambda)=\infty$, Niemi and Nummelin (\cite{ninu},Theorem 4) proved that there exist unique left and right invariant $\Lambda$-vectors as follows.

\begin{proposition}\label{invariant}
Assume that $H_{ij}(\Lambda)=\infty$ for some $i,j\in\N$. Then, 
\begin{enumerate}
	\item[i)] There exists a unique (up to scalar multiplication) positive left invariant $\Lambda$-vector.
\item[ii)] There exists a unique (up to scalar multiplication) positive right invariant $\Lambda$-vector. Moreover, any right subinvariant $\Lambda$-vector is a right invariant vector.
\end{enumerate}
\end{proposition}

From the previous propositions, there exists at least   a positive left (right) subinvariant $\Lambda$-vector. 
One of the reasons we are interested in right (sub)invariant vector, is that we can associate to it a (super)martingale, which will be of use later on in our analysis.
\begin{proposition}\label{martingale}
Let ${x}$ be a right subinvariant $\lambda$-vector. Then
$$W_t:= {\rm e}^{\lambda t} \langle {x},{X}_t\rangle, \qquad t\geq 0, $$
is a supermartingale. If ${x}$ is also an invariant vector, then $(W_t, t\geq 0)$ is a martingale.
\end{proposition}

\begin{proof}
Let $t,s\geq 0$. By the Markov property and the branching property
$$\mbf{E}\left[\left. {\rm e}^{\lambda(t+s)}\langle {x},X_{t+s}\rangle\right|\mathcal{F}_s\right]={\rm e}^{\lambda(t+s)}\mbf{E}_{X_s}\left[\langle {x},{X}_t\rangle\right]={\rm e}^{\lambda(t+s)}\underset{i\in\N}{\sum}\:X_s(i)\mbf{E}_{\delta_i}\left[\langle {x},{X}_t\rangle\right].$$
Since  ${x}$ is a right  subinvariant $\lambda$-vector, 
$$\mbf{E}_{\delta_i }\left[\langle  {x},{X}_t\rangle\right]=[{M}(t){x}](i)\leq {\rm e}^{-\lambda t}x(i),$$
therefore, we have that 
$$\mbf{E}\left[\left. W_{t+s}\right|\mathcal{F}_s\right]={\rm e}^{\lambda(t+s)}\underset{i\in\N}{\sum}\: X_s(i)[{M}(t){x}](i)\leq {\rm e}^{\lambda s} \underset{i\in\N}{\sum}\:x(i) X_s(i)=W_s.$$
In the invariant case, inequalities become equalities.
\end{proof}

Let $[n]=\{1,\cdots,n\}$ and let $X^{[n]}: = \{X_t^{[n]}: t\geq 0\}$ be a branching process with the same mechanism as $X_t$ but we kill mass that is created outside of $[n]$. 
To be more precise, $X^{[n]}$ has the same local branching mechanisms $\psi(i,\cdot)$ and $\phi(i,\cdot)$, for $i =1,\cdots, n$, albeit that, now, $\pi_i(j)$, $j\in\N\backslash\{i\}$ is replaced by $\pi_i(j)\mathbf{1}_{(j\leq n)}$, $j\in\N\backslash\{i\}$. Finally $\psi(i,\cdot)$ and $\phi(i,\cdot)$ are set to be zero  for $i\geq n$.
 
\noindent Let ${M}^{[n]}(t)$ be the matrix associated to the linear semigroup of $X^{[n]}$.  Then the infinitesimal generator of ${M}^{[n]}(t)$ is given by
$${L}^{[n]}=[{\Delta}_{-b}+{K}]\Big|_{[n]}. $$

In order to apply Perron-Froebenius theory to the matrix ${M}^{[n]}(t)$, we need  irreducibility. By Lemma \ref{irreducible}, it is enough that $\pi_i(j), i,j\leq n$ is irreducible. There exist simple examples of infinite irreducible matrices such that their upper left square $n$-corner truncations are not irreducible for all $n\geq 1$. However, according to Seneta (\cite{se1968}, Theorem 3), the irreducibility of $\pi$ implies that there exists a simultaneous rearrangement of the rows and columns of $\pi$, denoted by $\tilde{\pi}$, and a sequence of integers $k_n$ tending to infinity, such that the truncation of $\tilde{\pi}$ to $[k_n]$ is irreducible for all $n$. 
\noindent Observe that the type space, $\N$, is used as a labelled set and not as an ordered set. It therefore follows that  we can assume without loss of generality, that we start with $\tilde{\pi}$ (The vectors $b,c,d,\beta, \ell $ and ${\rm n}$ will require the same rearrangement). In the rest of the paper, when requiring finite truncations to the state space, whilst preserving irreducibility, it is enough to work with the truncations on $[k_n]$. In order to simplify the notation, we will assume without loss of generality that $k_n=n$ for all $n$. 
%
%

Classical Perron-Froebenius theory tells us there exist two positive vectors ${x}^{[n]}=\{x^{[n]}(i): i=1,\cdots, n\}$ and   ${y}^{[n]}=\{y^{[n]}(i): i = 1,\cdots,  n\}$, and a real number  $\Lambda^{[n]}=\sup\{\lambda\geq -\infty: H_{ij}^{[n]}(\lambda)<\infty\},$   such that
$${M}^{[n]}(t){x}^{[n]}= {\rm e}^{-\Lambda^{[n]} t}{x}^{[n]} \qquad \mbox{and}\qquad ({y}^{[n]})^{T}{M}^{[n]}(t)= {\rm e}^{-\Lambda^{[n]} t}{y}^{[n]}.$$
By construction of $X_t^{[n]}$, we have the inequalities  
$$M_{ij}^{[n]}(t)\leq M_{ij}^{[n+1]}(t)\leq M_{ij}(t),$$
which naturally leads to the hierarchy of eigenvalues
\begin{align}\label{deslambda}
\Lambda\leq \Lambda^{[n+1]}\leq \Lambda^{[n]}.
\end{align}


\begin{lemma}\label{lambdainfinito}$\mbox{ }$
	\begin{enumerate}
		\item[i)] $\Lambda^{{\infty}}:=\lim_{\ n\rightarrow\infty}\Lambda^{[n]}=\Lambda$.
		\item[ii)]  Let ${x}^{[n]}$ be a  right invariant $\Lambda^{[n]}$-vector for ${M}^{[n]}$, such that $x^{[n]}(1)=1$. Then, the vector $\{x^*(j):j\in\N\}$ given by $x^*(j)=\liminf_{n\rightarrow \infty}x^{[n]}(j)$ is a positive right $\Lambda$-subinvariant vector. Moreover, it $H_{ij}(\Lambda)=\infty$, then $\{x^*(j):j\in\N\}$ is the unique positive right invariant $\Lambda$-vector of $M$ with $x^*(1)=1$.
			\end{enumerate}
\end{lemma}
\begin{proof}
By inequality (\ref{deslambda}),  
$$\Lambda\leq\Lambda^{{\infty}}=\lim_{\ n\rightarrow\infty}\Lambda^{[n]}.$$
For any $n\in\N$, let ${x}^{[n]}$ be a ${M}^{[n]}$ right invariant vector, such that $x^{[n]}(1)=1$ for all $n\in \N$, this implies
$$ {L}^{[n]}{x}^{[n]}=-\Lambda^{[n]}{x}^{[n]}.$$
Let $x^*(j)=\liminf_{n\rightarrow \infty}x^{[n]}(j),$ by Fatou's Lemma
$$L x^* \leq-\Lambda^{{\infty}}{x}^*.$$
Using the fact that ${M}(t)$ is a non negative matrix and 
$$\frac{\ud}{\ud t}[{M}(t){x}^*](i)=[{M}(t)L x^*](i),\qquad i\in \N,$$
we find that
$$[{M}(t){x}^*](i) \leq {\rm e}^{-\Lambda^{{\infty}}t}{x}^*(i),\qquad i\in \N.$$
Since $x^*(1)=1$, $x^*$ is a right $\Lambda_{\infty}$-subinvariant vector. By applying Proposition \ref{subinvariant} we have that $\Lambda^{{\infty}} \leq \Lambda$  and therefore $x^*$ is a right $\Lambda$-subinvariant vector. The last part of the claim is true due to Proposition \ref{invariant}.
\end{proof}
\noindent Any vector ${x}\in\R^n$ can be extended to a vector $ u\in\R^{\N}$ by the natural inclusion map $u(i)=x(i)\Indi{i\leq n}$.  Since it will be clear in which space we intend to use the vector, we make an abuse of notation, and in the future we will denote both with ${x}$. 

\section{Spine decomposition}\label{spine1}
 
According to Dynkin's theory of exit measures \cite{dyn93} it is possible to describe the mass of $X$ as it first
exits the growing family of domains $[0,t)\times[n]$ as a sequence of random measures, known as {\it branching Markov exit measures}, which we denote by $\{X^{[n],t}: t\geq 0\}$.  Informally, the measure $X^{[n],t}$ is the distribution of the mass obtained by `freezing' the mass of the MCSBP when it is outside $[0,t)\times[n]$ for the first time. See \cite[Chapter 3]{dyn02} for  details of branching Markov exit measures. 
We recover here some of its basic properties. First, $X^{[n],t}$ has support on $(\{t\}\times[n])\cup ([0,t]\times [n]^c)$. Moreover, under $\{t\}\times[n],$ 
$$X^{[n],t}( \{t\}\times B)= X_t^{[n]}(B),$$
 for each $B\subset [n]$. 
 We use the obvious notation that for all ${f}\in\mca{B}^+([0,t]\times \N)$,
 \[
 \langle {f},X^{[n],t}\rangle=\sum_{i \in[n]} f(t,i)X^{[n],t}(\{ t\}, i) + \sum_{i\in[n]^c}\int_0^t f(s,i)X^{[n],t}(\ud s, i).
 \]

 We have that for all ${\mu}\in M([0,t]\times \N)$, and ${f}\in\mca{B}^+([0,t]\times \N)$
\begin{align}\label{exitmeasure1}
\mbf{E}_{\mu}[{\rm e}^{-\langle {f}, X^{[n],t}\rangle}]=\exp\{-\langle V^{[n],t}_0{f}, {\mu}\rangle\},
\end{align}
where, for $t\geq r\geq 0$, $V^{[n],t}_rf:[n]\rightarrow [0,\infty)$ is the unique non-negative solution to
\begin{equation}
V^{[n],t}_rf(i)= \left\{ \begin{array}{lcl}
f(t,i)-\int_r^{t}\big[{\psi}(i,{V_s^{[n],t}}f(i))+{\phi}(i,V^{[n],t}_sf)\big]{\ud} s & \mbox{ if } & i\leq n\\
& & \\
f(r,i) & \mbox{ if } & i>n.
\end{array}
\right.
\label{killedsgp}
\end{equation}
An important observation for later is that if the value of $f$ doesn't depend on time $t$ (temporal homogeneity), then 
\begin{equation}
V^{[n],t}_rf = V^{[n],t-r}_0f, \label{sphom}
\end{equation}
for all ${f}\in\mca{B}^+([0,t]\times \N)$. Moreover, as a process in time, $X^{[n],\cdot}=\{X^{[n],t}:t\geq 0\}$ is a MCSBP with local mechanism $\psi^{[n]}=\psi(i,z)\Indi{i\leq n}$ and non-local mechanism $\phi^{[n]}=\phi(i,f)\Indi{i\leq n}$.

Let denote by $\mathbf{N}_i$ the excursion measure of the $(\psi^{[n]},\phi^{[n]})$-MCSBP corresponding to $\mathbf{P}_{\delta_i}$. To be more precise, Dynkin and Kuznetsov (\cite{dynkuz}) showed that associated to the laws $\{\mathbf{P}_{\delta_i}: i\in\N\}$ are measures $\{\mathbf{N}_i:i\in\N\}$, defined on the same measurable space, which satisfy
$$\mathbf{N}_i(1-{\rm e}^{-\langle {f}, X^{[n],t}\rangle})=-\log\mathbf{E}_{\delta_i}({\rm e}^{-\langle {f}, X^{[n],t}\rangle}),$$
for all non-negative bounded function $f$ on $\N$ and $t\geq 0$. Intuitively speaking, the branching property implies that $\mathbf{P}_{\delta_i}$ is an infinitely divisible measure on the path space of $X^{[n],\cdot}$ and the previous equation is a ``L\'evy--Khinchine''
formula in which $\mathbf{N}_i$ plays the role of its ``L\'evy measure". A particular feature of $\mathbf{N}_i$ that we shall use later is that 
\begin{equation}
\mathbf{N}_i(\langle {f}, X^{[n],t}\rangle) = \mathbf{E}_{\delta_i}[\langle {f}, X^{[n],t}\rangle].
\label{N}
\end{equation}

Given two functions $x,y:\N\rightarrow [0,\infty)$ we denote by $x\circ y$ the element wise multiplication, $[x\circ y] (i)=x(i)y(i)$. Any function $g:\N\rightarrow[0,\infty)$ can be extended to a function $\bar{g}: [0,\infty)\times\N\rightarrow [0,\infty)$ such that $\bar{g}(s,i)=g(i)$.  

Let ${x}$ be a $\Lambda^{[n]}$ right invariant vector of ${M}^{[n]}$. (Note, in order to keep notation to a minimum, we prefer ${x}$ in place of the more appropriate notation ${x}^{[n]}$.)   By splitting the integral between $\{t\}\times[n]$ and $[0,t]\times [n]^c$, it is easy to show that
$$\langle \bar{x}, X^{[n],t}\rangle=\langle x, X^{[n]}_t\rangle.$$ 
Using the Markov property of exit measures, the last equality, and Proposition  \ref{martingale}, standard computations tell us that
$$Y_t^{[n]}:={\rm e}^{\Lambda^{[n]}t} \frac{\langle \bar{x},X^{[n],t}\rangle}{\langle x,{\mu}\rangle}={\rm e}^{\Lambda^{[n]}t} \frac{\langle x,X^{[n]}_t\rangle}{\langle x,{\mu}\rangle},\qquad t\geq 0,$$
is a mean one $\mbf{P}_{\mu}$-martingale. For ${\mu}\in\mathcal{M}(\N)$ such that $ \mu(\N\backslash [n])=0$, define $\widetilde{\mbf{P}}^{[n]}_{\mu}$ by the martingale change of measure

$$\frac{d\widetilde{\mbf{P}}^{[n]}_{\mu}}{d\mbf{P}_{\mu}}\Big|_{\mca{F}_t}=Y_t^{[n]}.$$

\begin{theorem}\label{spinesgpn}
	Let $\mu$ a finite measure with support in $[n]$ and $g\in\mca{B}^{+}(\N)$. Introduce the Markov chain $(\eta, \p_\cdot^{x})$ on $[n]$ with infinitesimal matrix,  $\tilde{L}^{[n]}\in M_{n\times n}$, given by 
$$\tilde{L}^{[n]}_{ij}=\frac{1}{x(i)}\left({\Delta}_{-b}+{K}_{ij}+\Ind_{\{i=j\}}\Lambda^{[n]}\right)x(j).$$

	If $X$ is a MCSBP, then
	\begin{align}
\widetilde{\mbf{E}}^{[n]}_{\mu}&\left[{\rm e}^{-\langle {f}, X^{[n],t}\rangle}\frac{\langle\bar{x}\circ \bar{g}, X^{[n],t}\rangle}{ \langle\bar{x}, X^{[n],t}\rangle} \right]=\mbf{E}_{\mu}\left[
{\rm e}^{-\langle {f}, X^{[n],t}\rangle}\right] \times\notag\\
&\e_{ x\mu}^{x}\Bigg[\exp\left\{-\int_0^t \left(2c(\eta_s)V_{0}^{[n],t-s}f(\eta_s)+\int_0^{\infty}u(1-{\rm e}^{-uV_{0}^{[n],t-s}f(\eta_s) })\ell(\eta_s,\ud u)\right)\ud s\right\}
\notag\\
&\hspace{7cm}\times
 g(\eta_t)\underset{s\leq t}{\prod}{\Theta}_{\eta_{s-},\eta_s}^{[n],t-s}\Bigg],
	\label{semigroupdecompn}\end{align}
	where the matrices $\{\Theta^{[n],s}: s\geq 0\}$, are given by
	$${\Theta}^{ [n],t}_{i,j}=\frac{\pi_{i}(j)\beta(i)}{
		[{\Delta}_{-b} + {K} + \Lambda^{[n]}I]_{i, j}
	}
	\int_0^{\infty} u ({\rm e}^{-u\langle V_{0}^{[n],t}{f},\, \pi_{i}\rangle}-1){\rm n}(i,\ud u) +1$$
	and $$\mathbb{P}_{ x\mu}^{x}(\cdot) = \sum_{i \in [n]} \frac{ x(i)\mu(i)}{\langle {x}, {\mu}\rangle}\mathbb{P}^{x}_i(\cdot),$$
with  an obviously associated expectation operator $\mathbb{E}_{ x\mu}^{x}(\cdot)$.
\end{theorem}
\noindent This theorem suggest that under $\widetilde{\mbf{P}}^{[n]}_{\mu}$, our process can decomposed into 2 parts. The first one is a copy of the original process and the second one can be related to some independent processes of immigration. As we will see after the proof, the process of immigration is governed by an immortal particle or spine along which two independent Poisson point process of mass immigration occur. The non-local nature of the branching mechanism induces a new additional immigration at each time the spine jumps. Moreover the distribution of this new immigration mass depends on where the spine jumped from and where it jumps to.

\begin{proof}
We start by noting that 
\begin{align*}
\widetilde{\mbf{E}}^{[n]}_{\mu}\left[{\rm e}^{-\langle {f}, X^{[n],t}\rangle}\frac{\langle\bar{ x}\circ \bar{g}, X^{[n],t}\rangle}{ \langle\bar{ x}, X^{[n],t}\rangle} \right]&=\frac{{\rm e}^{\Lambda^{[n]}t}}{\langle  x,{\mu}\rangle}\mbf{E}_{\mu}\left[\langle\bar{  x}\circ \bar{g}, X^{[n],t}\rangle
{\rm e}^{-\langle {f},\! X^{[n],t}\rangle}\right].
\end{align*} 
 Replacing $f$ by $f+\lambda \bar{ x}\circ \bar{g}$ in (\ref{exitmeasure1}) and (\ref{killedsgp}) and differentiating with respect to $\lambda$ and then setting $\lambda  =0$, we obtain
 \begin{align}
\widetilde{\mbf{E}}^{[n]}_{\mu}\left[{\rm e}^{-\langle {f}, X^{[n],t}\rangle}\frac{\langle\bar{x}\circ \bar{g}, X^{[n],t}\rangle}{ \langle\bar{x}, X^{[n],t}\rangle} \right]&=\mbf{E}_{\mu}\left[
{\rm e}^{-\langle {f},\! X^{[n],t}\rangle}\right]\frac
{\langle  \theta^t_0, {x}\circ{\mu}\rangle}{\langle x,{\mu}\rangle},\nonumber\\
&=\mbf{E}_{\mu}\left[
{\rm e}^{-\langle {f},\! X^{[n],t}\rangle}\right]\underset{i\leq n}{\sum}\frac{x(i)\mu_i}{\langle x,{\mu}\rangle}\theta^t_0(i),
\label{suma h}
\end{align}
where for $t\geq r\geq 0$, $
\theta^t_r$ is the vector with entries
\[
 \theta^t_r(i) := \frac{1}{x(i)}{\rm e}^{\Lambda^{[n]}(t-r)}\left.\frac{\partial}{\partial \lambda} V^{[n],t}_r[f+\lambda \bar{x}\circ \bar{g}](i)\right|_{\lambda  =0},  \qquad i\in[n].
\]  
So that, in particular,  $\theta^t_t(i) = g(i)$, $i\in[n]$, and additionally, $\theta^t_r(i) = 0$ for $i>n$ and $r\leq t$. Note that the temporal homogeneity property (\ref{sphom}) implies that $ \theta^t_r(i) =  \theta^{t-r}_0(i)$, $i\in[n]$, $t\geq r\geq 0$.
Moreover, $\theta^t_r(i)$, $i\in[n]$, is also the
 unique solution to
\begin{align*}
\theta^t_r(i) &=g(i)-\int_r^{t}\theta^t_s(i)\left[2c(i)V^{[n],t-s}_0f(i)+\int_0^{\infty}u(1-{\rm e}^{-uV^{[n],t-s}_0f(i)})\ell(i,\ud u)\right]\ud s \nonumber\\
&\hspace{2cm}+x(i)^{-1}\int_r^{ t}\big[( {\Delta}_{-b} + {K}+\Lambda^{[n]}{I})\textcolor{black}{{x}\circ \theta^t_s}\big](i)
\ud s\notag\\
&\hspace{2cm}+\int_r^{ t}
\langle \textcolor{black}{{{\theta}}^t_s},\, \pi^{x}_{i}\rangle\beta(i)\int_0^{\infty} u ({\rm e}^{-u\langle V^{[n],t-s}_0{f},\, \pi_{i}\rangle}-1){\rm n}(i,\ud u)\ud s,\nonumber\\
\end{align*}
where
\[
\pi_i^{x}(j): = \frac{x(j)}{x(i)}\pi_i(j), \qquad, i,j\in[n].
\]

A integration by parts now ensures that 
\begin{align*}
[{\rm e}^{\widetilde{L}^{[n]}r}{\theta}^t_r](i) =&[{\rm e}^{\widetilde{L}^{[n]}t}g](i)\\
&\ -
\int_r^{t}{\rm e}^{\widetilde{L}^{[n]}s}\left[{\theta}^t_s\circ\left[2c(\cdot)V^{[n],t-s}_0f(\cdot)+\int_0^{\infty}u(1-{\rm e}^{-uV^{[n],t-s}_0f(\cdot)})\ell(\cdot,\ud u)\right]\right](i)\ud s \nonumber\\
&\hspace{1cm}+\int_r^{ t}
{\rm e}^{\widetilde{L}^{[n]}s}\left[ 
\langle {\theta}^t_s,\, \pi^{x}_{\cdot}\rangle\beta(\cdot)\int_0^{\infty} u ({\rm e}^{-u\langle V^{[n],t-s}_0{f},\, \pi_{\cdot}\rangle}-1){\rm n}(\cdot,\ud u)\ud s\right](i)\ud s.\nonumber
\end{align*}
Then appealing to temporal homogeneity, 
and the fact that $\{{\rm e}^{\widetilde{L}^{[n]}t}:t\geq 0\}$ is the semigroup of $(\eta, \mathbb{P}^{x}_\cdot)$,
\begin{align*}
\theta^t_0(i) 
=&\mathbb{E}^{x}_i[g(\eta_t)]- \mathbb{E}^{x}_i\left[
\int_0^{t}\theta^{t-s}_0(\eta_s)\left[2c(\eta_s)V^{[n],t-s}_0f(\eta_s)+\int_0^{\infty}u(1-{\rm e}^{-uV^{[n],t-s}_0f(\eta_s)})\ell(\eta_s,\ud u)\right]\ud s
\right]\nonumber\nonumber\\ &\hspace{1.5cm}+\mathbb{E}^{x}_i\left[\int_0^{ t}
\langle {\theta}^{t-s}_0,\, \pi^{x}_{\eta_s}\rangle\beta(\eta_s)\int_0^{\infty} u ({\rm e}^{-u\langle V^{[n],t-s}_0{f},\, \pi_{\eta_s}\rangle}-1){\rm n}(\eta_s,\ud u)\ud s\right]\nonumber\\
=&\mathbb{E}^{x}_i[g(\eta_t)]- \mathbb{E}^{x}_i\left[
\int_0^{t}\theta^{t-s}_0(\eta_s)\left[2c(\eta_s)V^{[n],t-s}_0f(\eta_s)+\int_0^{\infty}u(1-{\rm e}^{-uV^{[n],t-s}_0f(\eta_s)})\ell(\eta_s,\ud u)\right] \ud s
\right]\nonumber\nonumber\\
& +\mathbb{E}^{x}_i\left[\int_0^{ t}\sum_j\mathbf{1}_{(\widetilde{L}^{[n]}_{\eta_s, j} \neq 0)}
\theta^{t-s}_0(j)\left(\frac{\pi^{x}_{\eta_s}(j)\beta(\eta_s)}{\widetilde{L}^{[n]}_{\eta_s, j}}\int_0^{\infty} u ({\rm e}^{-u\langle V^{[n],t-s}_0{f},\, \pi_{\eta_s}\rangle}-1){\rm n}(\eta_s,\ud u) \right)     \widetilde{L}^{[n]}_{\eta_s, j}\ud s\right]
\end{align*}
(Note, in the last equality, we have used that $\widetilde{L}^{[n]}_{\eta_s, j}=0$ if and only if $\pi_{\eta_s}(j) = 0$). We now see from Lemma \ref{nonlocal} in the appendix that 
{\small \begin{align*}
\theta^{t}_0(i)= &\e_{i}^{x}\left[\exp\left\{-\int_0^t \left(2c(\eta_s)V^{[n],t-s}_0f(\eta_s)+\int_0^{\infty}u(1-{\rm e}^{-uV^{[n],t-s}_0f(\eta_s) })\ell(\eta_s,\ud u)\right)\ud s\right\} \underset{s\leq t}{\prod}{\Theta}^{[n],t-s}_{\eta_{s-},\eta_s}\right],
\end{align*}}
as required.
\end{proof}

Fix  $\mu$ as  a finite measure with support in $[n]$. As we said before, Theorem \ref{spinesgpn} suggests that  the process $(X^{[n],\cdot},\widetilde{\mathbf{P}}_{\mu})$ is equal in law to a process $\{\Gamma_t : t\geq 0\}$,  whose law is henceforth denoted by ${\rm P}_{\mu}$, where
\begin{equation}
\Gamma_t = X'_t + \sum_{s\in D_t^{\rm c}} X^{{\rm c}, s}_{t-s} + \sum_{s\in D^{\rm d}_t}X^{{\rm d}, s}_{t-s} + \sum_{s\in D_t^{\rm j}}X^{{\rm j}, s}_{t-s}, \qquad t\geq 0,
\label{spine}
\end{equation}
such that  $X'$ is an independent copy of $(X^{[n],\cdot},\mathbf{P}_{\mu})$, the countable sets $ D^{\rm c}_\cdot,D^{\rm d}_\cdot, D^{\rm j}_\cdot $ and processes $X^{{\rm c}, s}_{\cdot}$, $X^{{\rm d}, s}_{\cdot}$ and $X^{{\rm j}, s}_{\cdot}$ are defined through a process of immigration as follows: Given the path of the Markov chain $(\eta, \mathbb{P}^{x}_{{x}\mu})$, 

\bigskip

\noindent {\bf [continuous immigration]}  in a Poissonian way  an $(\psi^{[n]},\phi^{[n]})$-MCSBP $X^{{\rm c},s}_{\cdot}$ is immigrated at $(s, \eta_s)$ with rate $\ud s\times 2c(\eta_s)d\mathbf{N}_{\eta_s}$. The almost surely countable
set of immigration times is denoted by $D^{\rm c}$ and $D^{\rm c}_t:=D^{\rm c}\cap(0,t]$,
\bigskip

\noindent {\bf [discontinuous immigration]} in a Poissonian way  an $(\psi^{[n]},\phi^{[n]})$-MCSBP $X^{{\rm d}, s}_\cdot$ is immigrated at $(s, \eta_s)$ with rate $\ud s\times \int_0^\infty u \ell(\eta_s, \ud u)\mathbf{P}_{u\delta_{\eta_s}}$. The almost surely countable
set of immigration times is denoted by $D^{\rm d}$ and $D^{\rm d}_t:=D^{\rm d}\cap(0,t]$,

\bigskip

\noindent {\bf [jump immigration]} at each jump time $s$ of $\eta$, an $(\psi^{[n]},\phi^{[n]})$-MCSBP $X^{{\rm j}, s}_\cdot$ is immigrated  at $(s, \eta_s)$
with law $\int_0^\infty \nu_{\eta_{s-},\eta_s}(\ud u)\mathbf{P}_{u\pi_{\eta_{s-}}}$, where, for $i,j$ in the range of $\eta$,

\[
\nu_{i,j}(\ud u) = \frac{[{\Delta}_{-b} +{I}\Lambda^{[n]}]_{i, j}\textcolor{black}{+\beta(i)d(i)\pi_i(j)}
}{[{\Delta}_{-b} + {K} + {I}\Lambda^{[n]}]_{i, j}
}\delta_0(\ud u) + \frac{\pi_{i}(j)\beta(i)}{[{\Delta}_{-b} + {K} + {I}\Lambda^{[n]}]_{i, j}
}u{\rm n}(i,\ud u).
\]
$D^{\rm j}$ denotes the set of jump times of $\eta$ and we denote by $D^{\rm j}_t$ the jump times before $t$.

\bigskip

\noindent Given $\eta$, all the processes are independent. We remark that we suppressed the dependence on $n$ of the processes $X',$ $X^{{\rm c}, s}_{\cdot}$, $X^{{\rm d}, s}_{\cdot}$, $X^{{\rm j}, s}_{\cdot}$ and $\Gamma $ in order to have a nicer notation. 

Observe that the processes $X^{\mbf{c}}$, $X^{\mbf{d}}$ and $X^{\mbf{j}}$ are initially zero valued, therefore, if $\Gamma_0=\mu$ then $X_0'=\mu$. Moreover  $(\eta,P_{\mu})$ is equal in distribution to $(\eta, \mathbb{P}^{x}_{{x}\mu})$. 
The following result corresponds to a classical spine decomposition, albeit now for the setting of an $(\psi^{[n]},\phi^{[n]})$-MCSBP. Note, we henceforth refer to the process $\eta$ as the {\it spine}.

\begin{remark}\rm
The inclusion of the immigration process indexed by $\rm j$ appears to be a new feature not seen before in previous spine decompositions and is a consequence of non-local branching. Simultaneously to our work, we learnt that a similar phenomenon has been observed by Chen, Ren and Song \cite{CRS}.
\end{remark}

\begin{theorem}[Spine decomposition]\label{gamma}
Suppose that $\mu$ as  a finite measure with support in $[n]$. Then $(\Gamma, {\rm P}_{\mu})$ is equal in law to $(X^{[n],\cdot}, \widetilde{\mathbf{P}}_{\mu})$.
\end{theorem}


\begin{proof}
The proof is designed in two steps. First we  show that $\Gamma$ is a Markov process. Secondly we show that $\Lambda$ has the same semigroup as $X^{[n],\cdot}$. In fact the latter follows immediately from Theorem \ref{spinesgpn} and hence we focus our attention on the first part of the proof. Observe that $((\Gamma_t,\eta_t), {\rm P}_{\mu})$ is a Markov process. By the same argument that appeared in Theorem 5.2 in \cite{klmr}, if we prove
\begin{align}
\label{condexplambda}
{\rm E}_{\mu}[\eta_t=i\mid \Gamma_t]=\frac{x(i)\Gamma_t(i)}{\langle \bar{{x}},\Gamma_t\rangle}, \qquad i\leq n,
\end{align}
then, $(\Gamma_t,{\rm P}_{\mu})$ is a Markov process.
By conditioning on $\eta$, using the definition of $\Gamma$, the equation \ref{semigroupdecompn} and the fact that $(\Gamma_t,{\rm P}_{\mu})$ is equal in law to  $(X_t, \widetilde{\mathbf{P}}_{\mu})$, for each $t$,  we obtain
$${\rm E}_{\mu}\left[{\rm e}^{-\langle {f}, \Gamma_t\rangle}g(\eta_t)\right]={\rm E}_{\mu}\left[{\rm e}^{-\langle {f}, \Gamma_t\rangle}\frac{\langle \bar{{x}}\circ \bar{g}, \Gamma_t\rangle}{\langle {x}, \Gamma_t\rangle}\right],\qquad \mbox{for all } {f},\ g \mbox{ measurables}.$$
The definition of conditional expectation implies (\ref{condexplambda}).
\end{proof}

%

\section{Martingale convergence}\label{spine2}

An important consequence of the spine decomposition in Theorem \ref{gamma} is that we can establish an absolute continuity between the measures $\mathbf{P}_{\mu}$ and $\widetilde{\mathbf{P}}^{[n]}_{\mu}$.

\begin{theorem}\label{mgcgcethrm} Fix $n\in\mathbb{N}$ and  ${\mu}\in\mathcal{M}(\mathbb{N})$ such that $\sup\{k: \mu(k)>0\}\leq n$.
The martingale $Y^{[n]}$ converges almost surely and in $L^1(\mathbf{P}_{\mu})$ if and only if 
  $\Lambda^{[n]}<0$ and that 
\begin{equation}
\sum_{i\in[n]}\int_1^\infty (x\log x) \ell(i,{\rm d}x)  + \sum_{i\in[n]}\int_1^\infty (x\log x) {\rm n}(i,{\rm d}x) <\infty,
\label{xlogx}
\end{equation}
Moreover, when these conditions fail, $\mathbf{P}_\mu(\lim_{t\to\infty}Y^{[n]}_t=0)=1$.
\end{theorem}

\begin{proof}
We follow a well established line of reasoning. Firstly we establish sufficient conditions. We know that $1/Y^{[n]}_t$ is a positive $\widetilde{\mathbf{P}}^{[n]}_{\mu}$-supermartingale and hence $\lim_{t\to\infty}Y^{[n]}_t$ exists $\widetilde{\mathbf{P}}^{[n]}_{\mu}$-almost surely. The statement of the theorem follows as soon as we can prove that $\widetilde{\mathbf{P}}^{[n]}_{\mu}(\lim_{t\to\infty}Y^{[n]}_t<\infty)=1$.

To this end,  consider the spine decomposition in Theorem \ref{gamma}. 
Suppose, given the trajectory of the spine $\eta$, that we write $(s, \Delta^{\rm d}_s, \Delta^{\rm j}_s)$, $s\geq 0$, for the  process of immigrated mass along the spine, so that $(s, \Delta^{\rm d}_s)$  has intensity 
$
{\rm d}s \times  u \ell(\eta_{s}, \ud u)
$
and, at $s$ such that $\eta_{s-}\neq \eta_s$, $ \Delta^{\rm j}_s$ is distributed according to $\nu_{\eta_{s-},\eta_s}$.
Let  $\mathcal{S} = \sigma(\eta, (s, \Delta^{\rm d}_s, \Delta^{\rm j}_s), s\geq 0)$  be the sigma algebra which informs the location of the spine and the volume of mass issued at each immigration time along the time
and write 
 \[
 Z^{[n]}_t = {\rm e}^{\Lambda^{[n]} t}\frac{\langle \bar{x}, \Gamma_t\rangle}{\langle{x}, {\mu}\rangle}.
 \]
Our objective now is to use Fatou's Lemma and show that  
\[
{\rm E}_{\mu}[\lim_{t\to\infty}Z^{[n]}_t | \mathcal{S}]\leq \liminf_{t\to\infty}{\rm E}_{\mu}[Z^{[n]}_t|\mathcal{S}]<\infty.
\]
 Given that $(\Gamma, {\rm P}_{\mu})$ is equal in law to $(X^{[n],\cdot}, \widetilde{\mathbf{P}}_{\mu})$, this ensures that $\widetilde{\mathbf{P}}^{[n]}_{\mu}(\lim_{t\to\infty}Y^{[n]}_t<\infty)=1$, thereby completing the proof. 

It therefore remains to show that $\liminf_{t\to\infty}{\rm E}_{\mu}[Z^{[n]}_t| \mathcal{S}]<\infty$. Taking advantage of the spine decomposition, we have, with the help of (\ref{N}) and the fact that $\mathbf{E}_{\mu}[Y^{[n]}_t] =1$, for  $t\geq 0$ and $\mu$ such ${\mu}\in\mathcal{M}(\mathbb{N})$ such that $\sup\{k: \mu(k)>0\}\leq n$,
\begin{align*}
\liminf_{t\to\infty}{\rm E}_{\mu}[Z^{[n]}_t|\mathcal{S}]& =\langle{x}, {\mu}\rangle +\int_0^\infty 2 c(\eta_s) {\rm e}^{\Lambda^{[n]}s}\frac{x_{\eta_s}}{\langle{x}, {\mu}\rangle}
{\ud}s\\
&
+\sum_{s>0}  {\rm e}^{\Lambda^{[n]}s}\Delta_s^{\rm d}\frac{x_{\eta_s}}{\langle{x}, {\mu}\rangle} 
+\sum_{s>0}  {\rm e}^{\Lambda^{[n]}s}\Delta_s^{\rm j}\frac{\langle{x}, \pi_{\eta_s-}\rangle}{\langle{x}, {\mu}\rangle} .
 \end{align*}
Recalling that $\Lambda^{[n]}<0$ and that $\eta$ lives on $[n]$, the first integral on the right-hand side above can be bounded above by a constant.  The two sums on the right-hand side above can be dealt with almost identically. 

It suffices to check that 
\begin{align}
\sum_{s>0}  {\rm e}^{\Lambda^{[n]}s}\mathbf{1}_{(\Delta_s^{\rm d}<1)}\Delta_s^{\rm d}&+ \sum_{s>0}  {\rm e}^{\Lambda^{[n]}s}\mathbf{1}_{(\Delta_s^{\rm j}<1)}\Delta_s^{\rm j}\notag\\
&+ \sum_{s>0}  {\rm e}^{\Lambda^{[n]}s}\mathbf{1}_{(\Delta_s^{\rm d}\geq 1)}\Delta_s^{\rm d}
+ \sum_{s>0}  {\rm e}^{\Lambda^{[n]}s}\mathbf{1}_{(\Delta_s^{\rm j}\geq 1)}\Delta_s^{\rm j}<\infty.
\label{2sums}
\end{align}
We first note that 
\begin{align*}
{\rm E}_{\mu}&\left[\sum_{s>0}  {\rm e}^{\Lambda^{[n]}s}\mathbf{1}_{(\Delta_s^{\rm d}<1)}\Delta_s^{\rm d}+ \sum_{s>0}  {\rm e}^{\Lambda^{[n]}s}\mathbf{1}_{(\Delta_s^{\rm j}<1)}\Delta_s^{\rm j} \right]
 \\
&= {\rm E}_{\mu}\left[\int_0^\infty  {\rm e}^{\Lambda^{[n]}s}\int_{(0,1)}u^2 \ell(\eta_{s}, \ud u){\ud s} \right]
+ {\rm E}_{\mu}\left[\int_0^\infty  {\rm e}^{\Lambda^{[n]}s}\int_{(0,1)}u^2 L^{[n]}_{\eta_{s-}, \eta_s}
\nu_{\eta_{s-}, \eta_s}(du)
{\ud s}
 \right]\\
&\leq \int_0^\infty  {\rm e}^{\Lambda^{[n]}s}\ud s \left\{\sup_{i\in [n]}\int_{(0,1)}u^2 \ell(i, \ud u){\ud s}
+ \sup_{i,j\in [n]} \pi^{x}_{i}(j)\int_{(0,1)}u^2{\rm n}(i, \ud u){\ud s}
\right\}
<\infty.
\end{align*}

Next, note that the condition (\ref{xlogx}) ensures that, ${\rm P}_{\mu}$ almost surely, 
\[
\limsup_{s\to\infty}s^{-1} \mathbf{1}_{(\Delta_s^{\rm d}\geq 1)}\log\Delta_s^{\rm d}
+\limsup_{s\to\infty}s^{-1} \mathbf{1}_{( \Delta_s^{\rm d}\geq 1)}\log\Delta_s^{\rm j}
=0,\]
so  that both sequences $\Delta_{s}^{\rm d}$ and $\Delta_{s}^{\rm j}$ in the last two sums of  (\ref{2sums}) grow subexponentially. (Note that both of the aforesaid sequences are indexed by a discrete set of times when we insist $\{\Delta_s^{\rm d}\geq 1\}$.) Hence the second term in (\ref{2sums}) converges.

To establish necessary conditions, let us suppose that $\tau$ is the set of times at which the mass $(s, \Delta^{\rm d}_s, \Delta^{\rm j}_s)$, $s\geq 0$, immigrates along the spine. We note that for $t\in\tau$,
\begin{equation}
Z^{[n]}_t\geq   {\rm e}^{\Lambda^{[n]}t}\Delta_t^{\rm d}\frac{x_{\eta_t}}{\langle{x}, {\mu}\rangle} 
+ {\rm e}^{\Lambda^{[n]}t}\Delta_t^{\rm j}\frac{\langle{x}, \pi_{\eta_t-}\rangle}{\langle{x}, {\mu}\rangle} .
\label{lower}
\end{equation}
If $\Lambda^{[n]}> 0$ and (\ref{xlogx}) holds then 
\begin{equation}
\widetilde{\bf P}_\mu(\limsup_{t\to\infty}Y^{[n]}_t = \infty) = {\rm P}_\mu(\limsup_{t\to\infty}Z^{[n]}_t = \infty)=1
\label{limsupy}
\end{equation}
 on account of the term $ {\rm e}^{\Lambda^{[n]}t}$, the remaining terms on the righ-hand side of (\ref{lower}) grow subexponentially. If  $\Lambda^{[n]}= 0$ and (\ref{xlogx}) holds then, although there is subexponential growth of $(\Delta_t^{\rm j}, \Delta_t^{\rm d})$, $t\geq 0$, 
\[
\limsup_{t\to\infty}\mathbf{1}_{(\Delta_s^{\rm d}\geq 1)}\Delta_s^{\rm d}
+\limsup_{s\to\infty}\mathbf{1}_{( \Delta_s^{\rm d}\geq 1)}\Delta_s^{\rm j} =\infty
\] nonetheless. This again informs us that (\ref{limsupy}) holds. Finally if $\Lambda^{[n]}<0$ but (\ref{xlogx}) fails, then there exists an $i\in[n]$ such that $\int_1^\infty (x\log x) \ell(i,{\rm d}x) =\infty$ or $\int_1^\infty (x\log x) {\rm n}(i,{\rm d}x)=\infty$. Suppose it is the latter. Recalling that $\eta$ is ergodic, another straightforward Borel-Cantelli Lemma tells us that 
\[
\limsup_{s\to\infty}s^{-1} \mathbf{1}_{(\eta_{s-} = i,\, \Delta_s^{\rm d}\geq 1)}\log\Delta_s^{\rm j} >c,
\]
for all $c>0$,
which implies superexponential growth. In turn,  (\ref{limsupy}) holds.
The proof of the theorem is now complete as soon as we recall that (\ref{limsupy}) implies that ${\bf P}_\mu$ and $\widetilde{\bf P}_\mu$ are singular and hence $\widetilde{\bf P}_\mu(\lim_{t\to\infty}Y^{[n]}_t =0)=1$.
\end{proof}

\section{Local and global extinction}\label{proofs}
\begin{lemma}\label{limsup}
For any finite $A\subset \N$ and any $\mu$,
$$\mbf{P}_{\mu}\left(\underset{t\rightarrow\infty}{\limsup} \langle\mathbf{1}_A,X_t\rangle\in\{0,\infty\}\right)=1.$$
\end{lemma}

\begin{proof}It is  enough to prove the lemma for $A=\{i\}$. The branching property implies that $X_1(i)$ is an infinitely divisible random variable and consequently, its distribution has unbounded support on $\R_+$, (see Chapter 2 in Sato \cite{SA}). Therefore, for all $\epsilon>0$,
\begin{align}
\label{edeltak}
\p_{\epsilon\delta_i}(X_1(i)>K)>0.
\end{align}
Let us denote by $\Omega_{0}$ the event $\limsup_{t\rightarrow \infty} X_t(i)>0$ and, for each $\epsilon>0$, denote by $\Omega_{\epsilon}$ the event $\limsup_{t\rightarrow \infty} X_t(i)>\epsilon$. Define the  sequence of stopping times as follows. 
On $\Omega_{\epsilon}$, let $T_0=\inf\{t> 0:X_t(i)\geq \epsilon\}$ and
 $T_{n+1}=\inf\{t>T_n+1:X_t(i)\geq \epsilon\}$ and for $\Omega_{\epsilon}^c$ 
 let $T_n=n$. Then, the $T_n$'s are finite stopping times on $\Omega_{\epsilon}$. 
Fix $K>0$ and let $A_n=\Omega_{\epsilon}\cap\{X_{T_n+1}(i)> K\}$ and $\Omega^1=\{\omega:\omega\in A_n\ \mbox{i.o.}\}$. Thus
by (\ref{edeltak}) and the strong Markov property,
$$\underset{n=1}{\overset{\infty}{\sum}}\p_{\mu}(A_n\mid X_{T_1},\cdots X_{T_n})=\infty \qquad \p_{\mu}\mbox{-a.s. on }\Omega_{\epsilon}.$$
By the extended Borel-Cantelli lemma [see Corollary 5.29 in \cite{Br}], $\p_{\mu}$-a.s. $\Omega_{\epsilon}\subset\Omega^1$. Observe that $\Omega_{\epsilon}\uparrow \Omega_0$ as $\epsilon\downarrow0$.  Therefore, for $K$ arbitrary large, $\limsup_{t\rightarrow \infty}X_t(i)\geq K$, $P_{\mu}$-a.s. on $\Omega_0$, and the claim is true. 
\end{proof}



Recall that we say that $X$ under $\mbf{P}_{\mu}$ \textit{exhibits local extinction} for the finite set $A\subset\mathbb{N}$ if 
$$\mbf{P}_{\mu}\left(\underset{t\uparrow\infty}{\lim} \langle \mathbf{1}_A, X_t\rangle=0\right)=1.$$
Now, we have all the preliminary results needed for the Proof of Theorem \ref{localextintion}

\begin{proof}[Proof of Theorem \ref{localextintion}]
(i)
 Let $0\leq\Lambda$. By Propositions \ref{subinvariant} and \ref{invariant}, there exists ${x}$ a positive right subinvariant $\Lambda$-vector. Proposition \ref{martingale} yields that $W_t= {\rm e}^{\Lambda t}\langle {x},{X}_t\rangle $ is a non-negative supermartingale. By Doob's convergence theorem, there is a non-negative finite random variable $W$ such that a.s.
 $$W_t\longrightarrow W \qquad \mbox{ as }\ t\rightarrow \infty.$$
When $\Lambda>0$, since 
 ${\rm e}^{\Lambda t}\rightarrow \infty$ as  $t\rightarrow \infty,$
 and $x(i)>0$ for any $i\in \N$, we have that $\mbf{P}_{\mu}$-a.s.
 $\lim_{t\rightarrow\infty} X_t(i) =0$, and hence, $\mbf{P}_{\mu}$-a.s.,
 $\lim_{t\rightarrow\infty}\langle\mathbf{1}_A, X_t\rangle =0$. 
 When $\Lambda=0$, Lemma \ref{limsup} yields the claim.

 (ii)  Now suppose that $\Lambda<0$, using Lemma \ref{lambdainfinito} there exits $n\geq i$ such that $\Lambda^{[n]}<0$.
 Next, consider the conclusion of Theorem \ref{gamma}. 
  Let $\boldsymbol 1$ be $n$-dimensional vector whose entries are all $1$ and let $\boldsymbol 0$ be similarly defined. Note that  $\widetilde{L}^{[n]}\boldsymbol{1} = \boldsymbol{0}$ and hence, together with irreducibility of $\left.\pi\right|_{[n]}$, it follows that $(\eta, \mathbb{P}^{x}_\cdot)$ is ergodic. As a consequence, of the spine decomposition (\ref{spine}), we now see that, $\widetilde{\mathbf{P}}^{[n]}_{\mu}$-almost surely,  mass is deposited by $\eta$ infinitely often in state $i$. Thanks to the assumption (\ref{globalxlogx}) and Theorem \ref{mgcgcethrm}, we have that $\widetilde{\mathbf{P}}^{[n]}_{\mu}\ll \mathbf{P}_{\mu}$ and hence there is no local extinction.

Next, recall that for a finite set of types $A\subset\mathbb{N}$ $$v_A(i)=-\log \mbf{P}_{\delta_i}(\mca{L}_A).$$ It is a trivial consequence of the fact that $\mathcal{E}\subseteq\mathcal{L}_A$ that $v_A(i)\leq w(i)$, $i\in\N$.
 By independence, it follows that, for all finite ${\mu}\in\mathcal{M}(\mathbb{N})$, 
$$\mbf{P}_{\mu}(\mca{L}_A)=\exp\left\{-\langle v_A,{\mu}\rangle \right\}, \qquad t\geq 0.$$
  By conditioning the event $\mca{L}_A$ on $\mca{F}_t,$ we obtain that for all $t\geq 0$,
\begin{align}\label{ecw}
 \mbf{E}_{\mu}({\rm e}^{-\langle v_A,{X_t}\rangle})=\exp\{-\langle v_A,{\mu}\rangle\}.
 \end{align}
 Now recalling \eqref{ecv}, $v_A$ must satisfy the semigroup evolution, see
\begin{align*}
 \psi(i,v_A(i))+\phi(i,v_A)=0.
\end{align*}
Formally speaking, to pursue the reasoning, we need $v_A$ to be a bounded vector, but this is not necessarily the case. To get round this problem, we can define $v_A^K(i) = K\wedge v_A(i)$, $i\in\N$, and observe by monotonicity and continuity that $V_tv_A^K(i)\uparrow v_A(i)$, $i\in\N, t\geq0$, as $K\uparrow\infty$. When seen in the context of \eqref{ecv} (also using continuity and monotonicity), the desired reasoning can be applied.
\end{proof}
\begin{proof}[Proof of Lemma \ref{global}] The proof that $w$ solves (\ref{fixedpoint}) is the same as the proof of \eqref{ecw}. 
\end{proof}

\section{Examples}\label{sect:examples}
This section is devoted to some examples, where we  find explicitly the global and local extinction probabilities. First we start with a remark of Kingman (see \cite{kingman}).

\begin{proposition}\label{king}
Let $P_{ij}(t)$ be the transition probabilities of an irreducible continuous-time Markov chain on the countable state space $E$. Then there exists $\kappa\geq 0$ such that for each $i,j\in E$, 
$$t^{-1}\log(P_{ij}(t))\rightarrow -\kappa.$$
Moreover, for each $i\in E$ and $t>0$
$$P_{ii}(t)\leq {\rm e}^{-\kappa t}$$
and there exist finite constants $K_{ij}$ such that
$$P_{ij}(t)\leq K_{ij} {\rm e}^{-\kappa t}, \qquad \mbox{for all } i,j\in E, \quad t>0.$$
If $Q=(q_{ij})$ is the associated $Q$-matrix, then
$$\kappa \leq -\sup\{q_{ii}: i\in E\}.$$
\end{proposition}

Observe that if the Markov chain is recurrent then $\kappa=0$. When it is transient, $\kappa$ could be greater than 0. In this case, we will say that the chain is geometrically transient with $\kappa$ its decay parameter. Kingman provided a random walk example where $\kappa>0$. The example is the following. Let $\xi$ a random walk with $Q$-matrix given by
$$q_{i,i=1}=p, \qquad q_{ii}=-1,\qquad q_{i,i-1}=q=1-p,$$
where $p\in(0,1)$. Then, $\xi$ is an irreducible process with decay parameter $\kappa=1-2\sqrt{pq}$. In particular, the process is geometrically transient except when $p=1/2$.

\bigskip

\noindent Now, we can provide some examples.

\bigskip

\noindent{\bf Example 1.}
If $\psi$, $\beta$, $d$ and $n$ don't depend on the underlying type, it is easy to show that $(\langle 1,X_t\rangle, t\geq 0)$ is a CSBP with branching mechanism given by
$$\widetilde{\psi}(z)=\left(b-\beta d-\beta \int_0^{\infty} u {\rm n}(\ud u)\right)z+cz^2+\int_0^{\infty}({\rm e}^{-zu}-1+zu)(\ell+\beta {\rm n})(\ud u), \qquad  z\geq0.$$
In this case, the global extinction probability is given by
$$\mbf{P}_{\delta_i}(\mca{E})={\rm e}^{-\widetilde{\Phi}(0)},$$
where $\widetilde{\Phi}(0)=\sup\{z\leq 0:\widetilde{\psi}(z)=0\}$.

Define $a=\beta d+\beta \int_0^{\infty} u {\rm n}(\ud u)$, then, our process $X$ has global extinction a.s. if and only if $b-a\geq 0$.
On the other hand, let $(\xi,\p_i)$ be an irreducible chain with $Q$-matrix given by
$$Q_{ij}=a(\pi_i(j)-\delta_{i=j}).$$
Then, by equation \eqref{f-k}, the linear semigroup of $X$ is
$$M_tf(i)=\e_i\left[f(\xi_t)\exp\left\{\int_0^t (a-b)(\xi_s){\rm d}s\right\}\right].$$
 In particular,
$$H_{ij}(\lambda)=\int_0^{\infty}{\rm e}^ {(\lambda +a-b)t}P_{ij}(t){\rm d}t.$$
If $(\xi,\p_i)$ is geometrically transient, then $\kappa\in(0,a)$ and $\lambda<b-a+\kappa$  implies $H_{ij}(\lambda)<\infty$. In particular if $a-\kappa<b$, the spectral radius of $M$ satisfies $\Lambda >0$ and, by Theorem \ref{localextintion}, $X$ presents local extinction a.s.

In summary, if $a-\kappa<b<a$ then the process presents local extinction a.s. but global extintion with probability less than one. 

\bigskip
\noindent{\bf Example 2.} Define $a(i)=\beta(i) d(i)+\beta(i)  \int_0^{\infty} u {\rm n}(i,\ud u)$.
Suppose that there exists a constant $c>0$ such that $b(i)-a(i)\geq c>0$. Let  $(\xi,\p_i)$ the associated irreducible chain in Lemma \ref{irreducible}. Let $0\leq \lambda <c$. By equation (\ref{f-k}) we have
$$H_{ij}(\lambda)=\int_0^{\infty} {\rm e}^{\lambda t} \e_i\left[\delta_j({\xi_t})\exp\left\{\int_0^t (a-b)(\xi_s){\rm d}s\right\}\right]{\rm d}t\leq \int_0^{\infty} {\rm e}^{(\lambda-c) t}{\rm d}t<\infty. $$
Then, $ \Lambda>0$ and the process presents local extinction a.s.

\bigskip

\noindent{\bf Example 3.} Suppose now that there exists a constant $c>0$ such that $b(i)-a(i)\leq -c<0$ and $(\xi,\p_i)$ is a recurrent Markov chain. Then, for $-c<\lambda<0$,
$$H_{ij}(\lambda )=\int_0^{\infty} {\rm e}^{\lambda t} \e_i\left[\delta_j(\xi_t)\exp\left\{\int_0^t (a-b)(\xi_s){\rm d}s\right\}\right]{\rm d}t\geq \int_0^{\infty}P_{ij}(t){\rm d}t =\infty. $$
It follows that $\Lambda<0$. If 
$$\sup_{i\in\mathbb{N}}\int_1^\infty (x\log x) \ell(i,{\rm d}x)  + \sup_{i\in\mathbb{N}}\int_1^\infty (x\log x) {\rm n}(i,{\rm d}x) <\infty,$$
then the process presents local extinction in each bounded  subset of $\N$ with probability less than one.
\section{Appendix}

We provide here a technical lemma pertaining to an extended version of the Feynman-Kac formula that is used in the main body of the text. Note that similar formulae have previously appeared in the literature e.g. in the work of Chen and Song \cite{CS}.

\begin{lemma}\label{nonlocal}
Let $(\xi_t, \p)$ be a Markov chain on a finite state space $E$ with $Q$ matrix ${Q}=(q_{ij})_{i,j\in E}$.  Let  $v:E\times\R_+\rightarrow\R$ be a measurable function and $F:E\times E\times\R_+\rightarrow \R$ be a Borel function vanishing on the diagonal of $E$. 
  For $i\in E$ and $t\geq 0$ and $f:E\rightarrow \R$, define 
  $$h(i,t):={\mathcal T}_t[f](i)=\e_i\left[f(\xi_t)\Exp{\int_0^t v(\xi_s,t-s)\ud s}\Exp{\sum_{s\leq t}F(\xi_{s-},\xi_s,t-s)}\right].$$
  Then ${\mathcal T}_t$ is a semigroup and for each $(i,t)\in E\times\R_+$, $h$ satisfies
  \begin{align}\label{feynmankac no local}
  h(i,t)=&\e_i\left[f(\xi_t)\right]+\e_i\left[\int_0^t h(\xi_s,t-s)v(\xi_s,t-s)\ud s\right]\nonumber\\
&+\e_i\left[\int_0^t \sum_{j\in E}h(j,t-s)({\rm e}^{F(\xi_s,j,t-s)}-1)q_{\xi_s, j}\ud s\right].
  \end{align}
  Moreover, if $v$ and $F$ do not depend on $t$, the semigroup has infinitesimal generator matrix ${P}$ given by,  
   \begin{align}\label{infge}
   p_{ij}=q_{ij}{\rm e}^{F(i,j)}+v(i)\Indi{i=j}.
   \end{align}
\end{lemma}

\begin{proof}
The Markov property implies the semigroup property. For each $0\leq s\leq t$ define 
$$A_{s,t}:=\int_s^t v(\xi_r,t-r)\ud r\sum_{s<r\leq t}F(\xi_{r-},\xi_r,t-r).$$
Then,
\begin{align*}
{\rm e}^{A_{0,t}}-{\rm e}^{A_{t,t}}&=\int_0^t v(\xi_s,t-s){\rm e}^{A_{s,t}}\ud s+\sum_{s\leq t}{\rm e}^{A_{s-,t}}({\rm e}^{F(\xi_{s-},\xi_s,t-s)}-1).
\end{align*}
This implies,
\begin{align*}
h(i,t)&=\e_i\left[f(\xi_t)\right]+\e_i\left[\int_0^t f(\xi_t) v(\xi_s,t-s){\rm e}^{A_{s,t}}\ud s\right]+\e_i\left[\sum_{s\leq t}f(\xi_t){\rm e}^{A_{s-,t}}({\rm e}^{F(\xi_{s-},\xi_s,t-s)}-1)\right].
\end{align*}
By the Markov property
\small{\begin{align*}
h(i,t)=\e_i\left[f(\xi_t)\right]+\e_i\left[\int_0^t  v(\xi_s,t-s)h(\xi_s,t-s)\ud s\right]+\e_i\left[\sum_{s\leq t}h(\xi_s,t-s)({\rm e}^{F(\xi_{s-},\xi_s,t-s)}-1)\right].
\end{align*}}
The L\'evy formula says that for any nonnegative Borel function $G$ on $E\times E\times \R_+$ vanishing on the diagonal and any $i\in E$,
$$\e_i\left[\sum_{s\leq t}G(\xi_{s-},\xi_s,s)\right]=\e_i\left[\int_0^t\sum_{y\in E}G(\xi_s,y,s)q_{\xi_s,y}\ud s\right].$$
Therefore, $h$ satisfies (\ref{feynmankac no local}). Using this expression, we can obtain the infinitesimal matrix.
  \end{proof}

\section*{Acknowledgements}
Both authors would like to thank Yanxia Ren for comments on an early draft of this paper.
The majority of this research was carried out during the 10 month stay at the University of Bath of SP.  During this time she was supported in part by CONACyT-MEXICO grant 351643 and in part by a Global Research Scholarship Scheme through the University of Bath Internationalisation Relations Office. SP would like to express her gratitude for this support and the hospitality of the Department of Mathematical Sciences. Both authors, as members of a Bath-CIMAT research group pairing, were additionally supported by a Royal Society Advanced Newton Fellowship. In addition, S.P. like to acknowledge support from a Royal Society Newton International Fellowship.


\begin{thebibliography}{99}

\bibitem{ashe} Asmussen, S.  and  Hering, H. (1983): {\it Branching processes.} Birkh\"auser. 


\bibitem{BLP}   Barczy, M.,  Li, Z.   and   Pap, G. (2015) Stochastic differential equation with jumps for multi-type continuous state and continuous time branching processes with immigration  {\it ALEA, Lat. Am. J. Probab. Math. Stat.} {\bf 12} 129-169.

\bibitem{BZ} Bertacchi, D., and Zucca, F. (2014) Strong local survival of branching random walks is not monotone. {\it Adv.  Appl. Probab.} {\bf 46} (2), 400-421.

\bibitem{Br} Breiman, L. (1992) {\it Probability.} Second edition. SIAM, Philadelphia.


\bibitem{CRS} Chen,  Z.Q., Ren, Y-X.,  and Song, R. (2017) $L\log L$ criterion for a class of multitype super diffusions with non-local branching mechanism.  \texttt{arXiv:1708.08219}


\bibitem{CS} Chen, Z.Q  and Song, R. (2003) Conditional gauge theorem for non-local
Feynman-Kac transforms. {\it Probab. Theory Relat. Fields} {\bf 125}, 45-72.



\bibitem{CLU} Caballero, M.E.,    Lambert, A.  and Uribe, G. (2009) Proof(s) of the Lamperti representation of continuous-state branching processes. {\it Probability Surveys}
{\bf 6}, 62-89. 


\bibitem{CGB} Caballero, M.E., P\'erez Garmendia, J.L. and Uribe Bravo, G. (2017) Affine processes on $\mathbb{R}^m_+\times \mathbb{R}^n$
and multi-parameter time changes.
{\it Ann. Inst. H. Poincar\'e,} {\bf 53} (3), 1208-1304.

 \bibitem{dagoli}   Dawson, D.,  Gorostiza, L.  and  Li, Z. (2002)  Non-local Branching superprocesses and some related models.  {\it Acta Appl. Math.} {\bf 74}: 93-112. 
 
 \bibitem{DFS} Duffie, D., Filipovic, D. and Schachermayer, W. (2003) Affine processes and applications in finance. {\it  Ann.  Appl. Probab.}
{\bf 13},  984-1053. 

\bibitem{dyn93} Dynkin, E. B. (1993) Superprocesses and Partial Differential Equations
{\it  Ann. Prob.} {\bf 21}, 1185-1262.

\bibitem{dyn04} Dynkin, E. B. (1994) {\it An introduction to Branching Measure-Valued processes.} CRM Monograph Series.



\bibitem{dyn02} Dynkin, E. B. (2002) {\it Diffusions, Superdiffusions and Partial Differential Equations.} American Mathematical Society Colloquium Publications 50. Amer. Math. Soc., Pro\-vi\-den\-ce.

\bibitem{dynkuz} Dynkin, E. B.,  Kuznetsov, S. E. (2004). $\N$-measures for branching exit Markov systems and their applications to differential equations.  {\it Probab. Theory Relat. Fields.} 130(1), 135-150.

\bibitem{E} Engl\"ander, J. (2015) {\it Spatial Branching in Random Environments and with Interaction}
        Advanced Series on Statistical Science and Applied Probability,  Volume 20, World Scien\-ti\-fic/Imperial College Press.


\bibitem{EK} Engl\"ander, J. and Kyprianou, A. E. (2004) Local extinction versus local exponential growth for spatial branching processes. {\it 
Ann. Probab.} {\bf 32}, 78-99. 


\bibitem{GT} Gabrielli, N. and Teichmann, J. (2014) Pathwise construction of affine processes. \texttt{arXiv:1412.7837}.

\bibitem{ERS} Engl\"ander, J., Ren, Y-X.  and Song, R.   (2016)
Weak extinction versus global exponential growth of total mass for super diffusions
{\it Ann. Inst. H. Poincaré Probab. Statist.} {\bf 52} (1), 448-482.

\bibitem{iknawa1} Ikeda, N.,  Nagasawa, M. and Watanabe, S.  (1968a) Branching Markov processes I   {\it J. Math. Kyoto Univ.} {\bf 8}-2,233-278                                 

\bibitem{iknawa2} Ikeda, N., Nagasawa, M.  and Watanabe, S. (1968b) { Branching Markov processes II} {\it J. Math. Kyoto Univ.} {\bf 8}-3, 365-410.

\bibitem{iknawa3} Ikeda, N., Nagasawa, M.  and Watanabe, S. (1969){ Branching Markov processes III}
{\it J. Math. Kyoto Univ.} {\bf 9}-1, 95-160.

\bibitem{J} Ji\v{r}ina, M.  (1958) Stochastic branching processes with continuous state space. {\it Czechoslovak Math. J.} {\bf 8}, 292-313.

\bibitem{kingman} Kingman, J. F. C. (1963). The exponential decay of Markov transition probabilities. {\it  Proc.  London Math. Soc.} {\bf  3}-1, 337-358.


\bibitem{K} Kyprianou, A.E. (2014) {\it Fluctuations of L\'evy Processes with Applications. 
Introductory Lectures.} Second Edition. Springer.


\bibitem{klmr} Kyprianou, A. E., Liu, R. L., Murillo-Salas, A., and Ren, Y. X. (2012). Supercritical super-Brownian motion with a general branching mechanism and travelling waves. {\it Ann. Inst. H. Poincar\'e}. {\bf 48}-3, 661-687.


\bibitem{L} Li, Z. (2010) {\it Measure-Valued Branching Markov Processes}. Springer.


\bibitem{M} Ma, R. (2013) Stochastic equations for two-type continuous-state branching processes with immigration. {\it Acta Math. Sin. (Engl. Ser.)} {\bf 29}, 287-294. 

\bibitem{Moy} Moy, S-T. (1967) Ergodic properties of expectation matrices of a branching process with countably many types. {\it J. Math. Mech.} {\bf 16}, 1207-1225.


\bibitem{ninu} S. Niemi and E. Nummelin. (1986) On non-singular renewal kernels with an application to a semigroup of transition kernel. {\it Stoch. Proc.  Appl.} {\bf 22}, 177-202.


\bibitem{SA} Sato, K. (1999) {\it L\'evy Processes and Infinitely Divisible Distributions.} Cambridge University Press.

\bibitem{se1968} Seneta, E. (1968). Finite approximations to infinite non-negative matrices, II: refinements and applications. {\it  Mathematical Proceedings of the Cambridge Philosophical Society} {\bf 64}(02), 465-470. Cambridge University Press.

\bibitem{se} Seneta, E. (2006). {\it Non-negative matrices and Markov chains.} Springer Science Business Media.


\bibitem{W} Watanabe, S. (1969) On two dimensional Markov processes with branching property
{\it Trans. Amer. Math. Soc.} {\bf 136} 447-466.







\end{thebibliography}
\end{document}